\documentclass[12pt]{amsart} 
 
\usepackage{amsmath} 
\usepackage{amscd}
\usepackage{amssymb} 
\numberwithin{equation}{section}

\input{prepictex} 
\input{pictex} 
\input{postpictex} 
\chardef\bslash=`\\ 
 
\makeatletter 
\def\verbatim{\interlinepenalty\@M \@verbatim 
  \leftskip\@totalleftmargin\advance\leftskip2pc 
  \frenchspacing\@vobeyspaces \@xverbatim} 
\makeatother 
\newtheorem{theorem}{Theorem}[section] 
\newtheorem{corollary}[theorem]{Corollary} 
\newtheorem{lemma}[theorem]{Lemma} 
\newtheorem{proposition}[theorem]{Proposition}

\theoremstyle{definition} 
\newtheorem{definition}[theorem]{Definition} 
\newtheorem{remark}[theorem]{Remark} 
 
\theoremstyle{main}
\newtheorem{main}{Theorem}
\theoremstyle{maincorollary}
\newtheorem{maincorollary}{Corollary} 
 
\newcounter{picture} 

\DeclareMathOperator{\conv}{conv} 
 
\newcommand{\FF}{{\mathbb F}} 
 
\newcommand{\KK}{{\mathbb K}}

\newcommand{\QQ}{{\mathbb Q}} 
\newcommand{\RR}{{\mathbb R}} 
 
\newcommand{\ZZ}{{\mathbb Z}} 

\newcommand{\cE}{{\mathcal E}} 
\newcommand{\cF}{{\mathcal F}}

\newcommand{\cS}{{\mathcal S}} 
\newcommand{\cT}{{\mathcal T}} 
 
\newcommand{\G}{{\Gamma}} 
\newcommand{\Om}{{\Omega}} 
 
\newcommand{\s}{{\sigma}}

\newcommand{\fS}{{\mathfrak S}}

\newcommand{\fX}{{\mathfrak X}}

\newcommand{\PGL}{{\text{\rm{PGL}}}}

\newcommand{\tq}{{\text{\rm{III}}}_{1/q}}
\newcommand{\tqs}{{\text{\rm{III}}}_{1/q^2}}

\newcommand{\pgl}{{\text{\rm{PGL}}}}

\begin{document} 
 
\def\frX        {{\mathfrak X}} 
\def\by         {{\bf y}} 
\def\bone       {{\bf 1}} 
\def\Proof      {{\it Proof}} 
\def\Aut        {\hbox{Aut}} 
\def\superset   {\supset} 
\def\bb         {{\bf b}} 
\def\GL         {{\text {\rm GL}}} 
\def\fS         {\cS} 
\def\type   {{\text {\rm type }}}

\title[]{Type III actions on boundaries of $\tilde A_n$ buildings}

\date{January 22, 2001}
\author{Paul Cutting}
\author{Guyan Robertson }
\address{Mathematics Department, University of Newcastle, Callaghan, NSW 
2308, Australia}
\email{cutting@maths.newcastle.edu.au}
\email{guyan@maths.newcastle.edu.au}
\thanks{This research was supported by The University of Newcastle} 
\thanks{ \hfill Typeset by  \AmS-\LaTeX}

\begin{abstract}
Let $\Gamma$ be a group of type rotating automorphisms of a building $\fX$ of type $\tilde A_n$ and order $q$. Suppose that $\G$ acts freely and transitively on the vertex set of $\fX$. Then the action of $\Gamma$ on the boundary of $\fX$ is ergodic, of type $\tq$ or type
$\tqs$ depending on whether $n$ is odd or even.
 
\end{abstract}      

\maketitle

\section*{Introduction} 
Let $M$ be a compact Riemannian manifold of negative sectional curvature, and let $\Gamma=\pi\sb 1(M)$. Then
$\Gamma$ acts on the sphere at infinity $S$ of the universal cover $\tilde M$ of $M$. The main result of \cite{Spa} is that the action of $\Gamma$
on $S$ is ergodic, amenable and type ${\text{\rm{III}}}_1$. This applies in particular to a cocompact Fuchsian group in $G=\pgl(2,\RR)$ acting on the circle.

A discrete analogue of this result was proved in \cite{RR1}.
Namely, let $\G$ be a free group acting simply transitively on the vertices of a locally finite homogeneous tree $\cT$ of degree $q+1$. Then $\cT$ is the universal covering space of a graph with fundamental group $\G$. It was shown in \cite{RR1} that the action of $\G$ on the boundary of the tree is ergodic, amenable and of type $\tq$.

Turning to higher rank spaces of nonpositive curvature, it is known that if $\Gamma$ is a lattice in
$G=\pgl(n+1,\RR)$ with $n\ge1$ and if $\Omega= G/B$ where $B$ is the Borel subgroup of upper
triangular matrices in $G$, then the action of $\Gamma$ on $\Omega$ is ergodic of type~${\text{\rm{III}}}_{1}$. Here $\Omega$ is the maximal boundary of Furstenberg \cite[VI.1.7]{mar}. A similar result holds more generally for a lattice $\Gamma$ in any semisimple noncompact Lie group $G$ \cite[4.3.15]{Zim}.

The discrete analogue of this construction is obtained by replacing $\RR$ by a nonarchimedean local field $\FF$ with residue field of order $q$.
The affine Bruhat-Tits building $\fX$ of $G=\PGL(n+1,\FF)$ is a building of type $\tilde A_n$ \cite{Steg}. 
The vertex set of $\fX$ may be identified with the homogeneous space $G/K$, where $K$ is a maximal compact subgroup of $G$, and $G$ acts on the boundary $\Omega= G/B$, where $B$ is a Borel subgroup of $G$. 

The precise higher rank analogue of the setup in \cite{RR1} is as follows. Let $\Gamma$ be a group of type rotating automorphisms of a building $\fX$ of type $\tilde A_n$, and suppose that $\Gamma$ acts simply transitively on the vertices of $\fX$. In view of the fact that $\tilde A_1$ buildings are trees, such groups $\Gamma$ should be regarded as higher rank analogues of free groups.
Note however that not every $\tilde A_2$ building $\fX$ is the Bruhat-Tits building of $\pgl\left(3,\KK\right)$ where $\KK$ is a local field~\cite[II \S8]{CMSZ}. Geometrically, an $\tilde A_n$ building $\fX$ is an $n$-dimensional contractible simplicial complex in which each codimension one simplex lies on $q+1$ maximal simplices ({\em chambers}). If $n\ge 2$ then the number $q$ is necessarily a prime power and is referred to as the {\em order} of the building.  The boundary $\Om$ of $\fX$ is a totally disconnected compact Hausdorff
space and is endowed with a natural family of mutually absolutely continuous Borel probability measures.  In \cite{RR} it was proved that, if $n=2$ and  $q\geq 3$, then the action of $\Gamma$ on $\Omega$ is ergodic and of type~$\tqs$. The purpose of the present article is to remove both these hypotheses and prove the following general result.

\begin{main}\label{main}
Let $n\ge 2$ and let $\fX$ be a locally finite thick $\tilde A_n$ building of order
$q$. Let
$\Gamma$ be a group of type rotating automorphisms of $\fX$ which acts simply and transitively on the vertices of $\fX$.  Then the action of $\Gamma$ on the boundary $\Omega$ of $\fX$ is amenable, ergodic and of $\type III_{\lambda}$, where
\begin{equation*}
\lambda = 
\begin{cases}
    1/q &  \text{if $n$ is odd},\\
    1/{q^2} &  \text{if $n$ is even}.
\end{cases}
\end{equation*}
\end{main}

The proof of this result will be completed in Section \ref{classification}. In Section \ref{freeness} we deal with freeness of the action, which is required in order to prove that the associated von Neumann algebra is a factor. In particular, Section \ref{freeness} removes a gap in the proof of freeness in \cite{RR}. We therefore obtain the following consequence.

\begin{maincorollary}
Let $\Gamma$ and $\Omega$ be as above. Then the crossed product von Neumann algebra
$L^\infty(\Omega)\rtimes\Gamma$ is the AFD factor of $\type III_{\lambda}$, where
$\lambda = 1/q$ if $n$ is odd, and 
$\lambda = 1/{q^2}$ if $n$ is even.
\end{maincorollary}

A simple variation on the arguments leading to Theorem \ref{main} proves the following result: see subsection \ref{classification_subsection}.

\begin{main}\label{main*}
Let $p\ge 2$ be a prime number, let $n\geq 1$, and let  $\Omega$ be the boundary of the affine building of $\PGL(n+1,\QQ_p)$. That is $\Omega=\PGL(n+1,\QQ_p)/B$, where $B$ is the Borel subgroup of upper triangular matrices.
Then the action of $\PGL(n+1,\QQ)$ on $\Omega$ is  ergodic and of $\type III_{\lambda}$, where
\begin{equation*}
\lambda = 
\begin{cases}
    1/p &  \text{if $n$ is odd},\\
    1/{p^2} &  \text{if $n$ is even}.
\end{cases}
\end{equation*}
\end{main}
Similar results can be stated for linear groups over other local fields, but this is perhaps the most striking case.
Note that, in contrast to Theorem \ref{main}, $\PGL(n+1,\QQ)$ is not a lattice in $\PGL(n+1,\QQ_p)$, and its action on the boundary is not amenable.

Given an~$\tilde A_n$ building~$\fX$, there is a type map~$\tau$ defined on the
vertices of~$\fX$ such that $\tau(v)\in\ZZ/(n+1)\ZZ$ for each vertex $v\in\fX$.
Every chamber of $\fX$ has precisely one vertex of each type.
An automorphism~$\alpha$ of~$\mathfrak X$ is said to be {\em type-rotating} if
there exists $i\in\{0,1,\ldots,n\}$ such that $\tau(\alpha v)=\tau(v)+i$ for all
vertices $v\in\mathfrak X$. An $\tilde A_1$ building is a tree, with two types of vertices,
and every automorphism of the tree is type rotating. We shall refer to a group $\Gamma$ satisfying the hypotheses of Theorem 
\ref{main} as an $\tilde A_n$ group.  
In \cite{C95} it was shown
that there is a 1-1 correspondence between $\tilde A_n$ groups and  ``triangle presentations''.  The right Cayley graph of an $\tilde A_n$ group $\Gamma$ relative to a natural set of generators is the $1$-skeleton of the $\tilde A_n$ building $\fX$. We shall frequently refer  to \cite{C99}, which lays much of the groundwork for dealing with the higher rank $\tilde A_n$ buildings.  

Throughout the paper $\fX$ will denote a thick, locally finite $\tilde A_n$ building, 
and the vertices of the building will be denoted by 
$\fX^0$.  If $\fX$ is associated with the $\tilde A_n$ group $\Gamma$  then the underlying 
set of the group $\Gamma$ will be identified with $\fX^0$, and the action of $\Gamma$ on the building will be by left
multiplication.  The identity of $\Gamma$ will be denoted by $1$ throughout.
For $x$, $y\in\fX^0$, $d(x,y)$ will denote the graph distance between 
those vertices in the 1-skeleton of $\fX$, and $|x|=d(x,1)$.

Further information on buildings can be found in \cite{C97}, \cite{Steg},
\cite{Brown} and \cite{Ronan}.  The first two of these references are introductory, 
while the last two provide a fuller account of the theory of buildings.

\section{Preliminaries}

This section mainly recalls material from \cite{C99}, to which we refer for a more complete discussion.  An $\tilde A_n$ building is a union of {\em apartments}.
An apartment is isomorphic to a Coxeter 
complex  of type $\tilde A_n$.    Let
$\Sigma$ denote the Coxeter complex of type $\tilde A_n$.
The $n$-simplices of $\Sigma$ are referred to as {\em chambers} and can be regarded as forming a tessellation of $\RR^n$.
The vertex set of $\Sigma$ 
can be identified with $\ZZ^{n+1}/\ZZ(1,1,\ldots,1)$. 
Two vertices $[a], [b]\in\Sigma$, $[a]=a+\ZZ(1,1,\ldots,1)$  
and $[b]=b+\ZZ{(1,1,\ldots,1)}$, are adjacent if there exist  
representative vectors 
$(a_1,a_2,\ldots,a_{n+1})\in[a]$ and $(b_1,b_2,\ldots,b_{n+1})\in[b]$ such  
that $a_i\leq b_i\leq a_i+1$ for all $1\leq i\leq n$. 
The type $\tau(x)\in \ZZ/(n+1)\ZZ$ of a vertex $[x]=[(x_1,x_2,\ldots,x_{n+1})]\in \Sigma$ is given by 
\[ 
\tau(x)=\left(\sum_i x_i\right)\mod(n+1)\,. 
\] 
Each chamber of $\Sigma$ 
has precisely one vertex of each
type.  
 
Let $\bb_i=(0,\ldots,0,1,\ldots,1)$, where  precisely 
$i$ entries equal $1$.  Note that each 
$x\in\ZZ^{n+1}$ can be written as 
\[
x=x_1{(1,1,\ldots,1)}+\sum(x_{i+1}-x_i){\bf b}_i.
\]
Hence there is a mapping  
$\ZZ^{n+1}/\ZZ{(1,1,\ldots,1)}\to\ZZ^{n}$ defined by  
\[
[x]\mapsto (x_{2}-x_1, x_3-x_2,\ldots,x_{n+1}-x_n).  
\]
This mapping is a 
canonical group homomorphism between $\ZZ^{n+1}/\ZZ{(1,1,\ldots,1)}$ and $\ZZ^n$, and by means of it the vertices of $\Sigma$ can be coordinatized by $\ZZ^n$.
Throughout this paper, if $k\in\ZZ^n$, then $k_i$ denotes the $i^{{\text{\it th}}}$
entry of $k$.
 
\subsection{The Boundary of an $\tilde A_n$ Building} 

Given an $\tilde A_n$ building $\fX$, one can define the boundary
of $\fX$ by means of equivalence classes of sectors. 
The central concern of this paper 
is the boundary regarded as a measure space.  For a discussion of the geometric structure of the boundary, the reader is referred to \cite[Chapter 9, 10]{Ronan}. 
 
Let $\cS_0$ be the simplicial cone  in the $\tilde A_n$ Coxeter complex $\Sigma$ with vertex  set coordinatized by $\ZZ_+^n$.
A subcomplex $S$ of $\fX$ is called a sector if there is an  
apartment $A$ containing $S$ and a type-rotating isomorphism $\phi:A\to\Sigma$ 
such that $\phi(S)\mapsto \cS_0$.  (Recall that the isometry $\phi$ is said to be {\em type rotating} if there exists $j\in \ZZ / (n+1)\ZZ$ such that, for each vertex $v$ of $S$, $\tau(\phi(v))=\tau(v)+j \pmod {n+1}$.
Note that if $a$, $b$ are vertices in a sector $S$ of $\fX$, and $\phi:S\to\cS_0$ is a type 
preserving isomorphism 
 such that $\phi(a)=(0,0,\ldots,0)$ and  
$\phi(b)=(k_1,k_2,\ldots,k_n)$, where $k_i\in\ZZ_+$, 
then the $k_i$ do not depend on the particular apartment $A$ 
containing $S$ \cite[Lemma 2.3]{C99}.  Thus, for $x\in\fX^0$ 
and $k\in\ZZ_+^n$, one can
define a set $S_k(x)$ consisting of those elements $y\in\fX^0$ such
that there exists a sector $S$ containing $x$ and $y$,
and a type rotating isomorphism
$\phi:S\to\cS_0$ such that $\phi(x)=(0,0,\ldots,0)$ and  $\phi(y)=k$.
Given a sector $S$ and a type rotating isomorphism with $\phi(S)=\cS_0$, 
the basepoint of $S$ is $v=\phi^{-1}(0,0,\ldots,0)$.
If $x,y\in \fX$ and $y\in S_k(x)$, with
$k=(k_1,k_2,\ldots,k_n)\in\ZZ_+^n$, then the graph distance between $x$ and $y$ is given by $d(x,y)=\sum_i k_i$.

Two sectors $S_1$, $S_2$ are said to be {\it equivalent} if  $S_1\cap S_2$  
contains a subsector.  Let $\Omega$ be the set of all such  
equivalence classes of sectors.  Then $\Omega$ is called the 
{\it boundary} of $\fX$.  Given $\omega\in\Omega$ and $x\in\fX^0$, there 
exists an unique sector with basepoint $x$ which is contained 
in the equivalence class $\omega$ \cite[Lemma 9.7]{Ronan}.  Denote 
this sector by $[x,\omega)$.  Also, for $m\in\ZZ_+^n$, let
$s^x_m(\omega)$ be the unique element in
the intersection $S_m(x)\cap [x,\omega)$. 
What Follows Is Based On \cite{C99}.

\begin{lemma}\label{rn4} 
Let $\omega\in\Omega$, and let $x$, $y\in\fX^0$.   
Then there exists ${m}(x,y;\omega)\in\ZZ^n$
such that  
$$s^x_k(\omega)=s^y_{k'}(\omega)\qquad {\rm where} \quad k'=k+m(x,y;\omega)\,,$$ 
for all $k\in\ZZ_+$ such that $k_i+m_i(x,y;\omega)\geq 0$ for 
$1\leq i\leq n$.
\end{lemma} 
 
\Proof. (c.f. \cite[Lemma 2.1]{CMS}.) Since $[x,\omega)$ is in the same equivalence class as $[y,\omega)$,  
$[x,\omega)\cap [y,\omega)$ contains a subsector.  Choose  
$$u=s_k^x(\omega)=s_{k'}^y(\omega)\in [x,\omega)\cap [y,\omega)\,.$$ 
Let $T=\{s^x_{k+l}(\omega);l\in\ZZ_+^n\}$, and 
$T'=\{s^y_{k'+l}(\omega);l\in\ZZ_+^n\}$. 
Then $T$, $T'$ are sectors in the equivalence 
class $\omega$ with a common base point, 
and so by \cite[Lemma 9.7]{Ronan}, $T=T'$.  It 
follows that $s^x_{k+l}(\omega)$, $s^y_{k'+l}(\omega)$ are both 
in $S_l^u(\omega)\cap T$, and hence are equal.  Thus 
${m}(x,y;\omega)=k'-k$, and $m(x,y;\omega)$ is clearly independent of 
the choice of $u\in[x,\omega)\cap[y,\omega)$. \qed 

\begin{lemma}\label{P3} 
Let $x$ be a vertex of $\fX$, and let $C$ be a chamber containing 
$x$.  Then for $\omega_0\in\Omega$, there exists an apartment  
$A$ which contains $C$ and the sector $S=[x,\omega_0)$. 
\end{lemma} 
 
\Proof.  By \cite[Lemma 9.4]{Ronan}, given the chamber $C$ and
sector $[x,\omega_0)$, there exists an apartment $A$ containing 
a subsector $S'\subset [x,\omega_0)$
and the chamber $C$.  Note that as $x\in C$, one has $x\in A$.  

Choose a sector $S''$ in $A$ with base vertex $x$ and parallel to $S'$.  Then $S''$ is equivalent to 
$[x,\omega_0)$ and so $S''=[x,\omega_0)$, by uniqueness of the sector with base vertex $x$ 
representing the boundary point $\omega_0$.
\qed

\medskip

The next lemma is a generalisation of \cite[Corollary 2.3]{CMS}.

\begin{proposition}\label{P1}
For $x$, $y\in\frX^0$, and $\omega\in\Omega$, one has
$$s^x_k(\omega)\in [x,\omega)\cap [y,\omega)$$
if $k_i\geq d(x,y)$ for $1\le i\le n$.
\end{proposition}

\Proof.  Set $r=d(x,y)$, and let $k=(r,r,\ldots,r)$.  An easy consequence
of Lemma \ref{rn4} is that
$z\in [x,\omega)$ implies $[z,\omega)\subset [x,\omega)$,
and so it is sufficient to show that 
$s^x_{k}\in [x,\omega)\cap [y,\omega)$.  

To proceed inductively, the case $d(x,y)=1$ is established first.
By Lemma \ref{P3}, there exists an apartment $A$ containing both
$y$ and $S=[x,\omega)$.  As $S$ is a sector, there exists 
a type rotating isomorphism $\varphi:A\to\Sigma$ such that 
$\varphi(S)=\cS_0$ with $\varphi(x)=0$.   
 
Since $x$, $y$ are adjacent in $A$, 
$\varphi(y)=(y_1,y_2,\ldots,y_n)$, where $y_i\in\{-1,1,0\}$,
 $0\leq i\leq n$.  Next, define the 
type rotating isomorphism $\phi:A\to\Sigma$ by
$$\phi(z)=\varphi(z)-\varphi(y)\,.$$
This map takes $y$ to the origin in $\Sigma$, and 
$(\phi)^{-1}(\cS_0)$ is a sector.  Moreover, for $z\in s_k^x(\omega)$,  
$\phi(z)=\varphi(z)-\varphi(y)=((k_1-y_1),(k_2-y_2),\ldots,(k_n-y_n))$. 
Thus 
$\phi(z)\in\cS_0$ if and only if $k_i\geq y_i$ for all $1\leq i\leq n$. 
 
It follows that 
$\phi^{-1}(\cS_0)=[y,\omega)$.  Moreover, as  
$(1,1,\ldots,1)\geq (y_1,\ldots,y_n)$, one has that  
$s^x_{(1,1,\ldots,1)}(\omega)\in [x,\omega)\cap [y,\omega)$.
This proves the case for $d(x,y)=1$.  

In general, given $s\in \ZZ_+, s>1$, suppose that the statement of the 
lemma is true for all $y'$ such that
$d(x,y')\leq s-1$, and let $y\in S_k(x)$ with $d(x,y)=s$.   
Without loss of generality, suppose
that $k_1\geq 1$ and set $k'=(k_1-1,k_2,\ldots,k_n)$.  
Let $z$ be the unique element in $\conv(x,y)\cap S_{k'}(x)$ and 
note that $d(z,y)=1$ and $d(z,x)=d(y,x)-1=s-1$.  Hence by the 
inductive hypothesis
$$a=s^x_{(s-1,s-1,\ldots,s-1)}(\omega)\in [x,\omega)\cap [z,\omega)\,.$$
Then for some $t=(t_1,t_2,\ldots,t_n)\in\ZZ_+^n$,   
one has $a=s^z_{t}(\omega)$ and $m_i(x,z;\omega)=(t_i-(s-1))$.  
As $t_i\in\ZZ_+$, it follows that 
$t_i+1\geq d(y,z)=1$.  By the inductive hypothesis, this implies that
$$s^z_{(t_1+1,\ldots,t_n+1)}(\omega)\in [z,\omega)\cap [y,\omega)\,.$$
Writing $t'=t-m(x,z;\omega)$,
$$s^z_{(t_1+1,t_2+1,\ldots,t_n+1)}(\omega)=
    s^x_{(t'_1+1,t'_2+1,\ldots,t'_n+1)}=s^x_{(s,\ldots,s)}(\omega)
    \in [x,\omega)\cap [y,\omega)\,.$$
The result follows. 
\qed

\begin{definition}\label{rn2}
Given $y\in\fX^0$, the topology on $\Omega$ based at $y$ is given
by the basis of open sets $\{\Omega^x_y\}_{x\in\fX^0}$, where
$$\Omega_y^x=\{\omega\in\Omega;x\in[y,\omega)\}\,.$$
\end{definition}
The topology so defined is independent 
of the choice of $y$. See below for details.
Note that for $y\in\fX^0$, and $k\in\ZZ_+^n$, the boundary $\Omega$ can
be expressed as the disjoint union
$$\Omega=\bigcup_{x\in S_k(y)}\Omega_y^x\,.$$

There is  a natural class of Borel measures on $\Omega$.  
Namely, for a fixed $y\in\fX^0$ and a basic open set $\Omega^x_y$ with
$x\in S_k(y)$, let
$$\nu_y(\Omega_y^x)=\frac{1}{|S_k(y)|}\,.$$
$|S_k|=|S_k(y)|$ is independent of $y$ and its actual value was determined 
in \cite[Corollary 2.7]{C99}.  
Specifically, let $q$ be the order of the $\tilde A_n$ building.  Also, 
for $k=(k_1,k_2,\ldots,k_n)\in\ZZ_+^n$, index the non-zero entries of
$k$ by
$\{i:k_i\geq 1\}=\{j_1,\ldots,j_t\}$, and set $j_0=0$ and $j_{t+1}=n+1$.  
Then
\begin{equation}\label{nastyformula}
|S_k|=q^{-\sum_{v=1}^t j_v(j_{v+1}-j_v)}\left[\begin{matrix}
                    n+1\\
                j_1-j_0,\ldots,j_{t+1}-j_t\\
                \end{matrix}\right]_q
    \cdot q^{\sum_{i=1}^ni(n+1-i)k_i}\,,
\end{equation}
where $[\cdots]_q=[n+1]_q/([j_1-j_0]_q\cdots [j_{t+1}-j_t]_q)$, and
$[k]_q=(q^k-1)\cdots(q-1)$.

Unlike the topology on $\Omega$, the value of the measure $\nu_y$ is
dependent on the choice of $y\in\fX^0$.  However, as shown
by the following lemmas, the set of measures $\{\nu_y\}_{y\in\fX^0}$
is absolutely continuous.
 
The next lemma is generalized from \cite[Lemma 2.4]{CMS}.

\begin{lemma}\label{rn5} 
Let $y\in S_k(x)$.  Suppose that $z\in S_l(x)\cap S_{l'}(y)$, 
where $l_i\geq d(x,y)$ for all $i$.  Then 
$\Omega_x^z\subset \Omega_y^z$. Moreover, if  
${m}(x,y;\omega)=(m_1,\ldots,m_n)$, as in Lemma \ref{rn4}, then  
$m_i(x,y;\omega)=l'_i-l_i$ for all $\omega\in\Omega_x^z$ 
\end{lemma} 
 
\Proof.  Let $\omega\in\Omega_x^z$. Then $z = s_l^x(\omega)$, 
and so $z$ is  an element of $[x,\omega)\cap [y,\omega)$ by Lemma \ref{P1}  
and choice of $l\in\ZZ_+^n$.  Thus $\omega\in\Omega_y^z$.  Moreover, by the 
proof of Lemma \ref{rn4}, $m_i(x,y;\omega)=l'_i-l_i$.   
\qed  

\begin{lemma}\label{rn6} 
The topology on $\Omega$ does not depend on the 
vertex $y\in\fX^0$ chosen in Definition~\ref{rn2}. 
For any $x$, $y\in\fX^0$, the measures $\nu_x$, 
$\nu_y$ are mutually absolutely continuous, and the 
Radon Nikodym derivative of $\nu_y$ with respect to $\nu_x$ is given by  
\begin{equation}\label{RND}
\frac{d\nu_y}{d\nu_x}(\omega)=q^{-\sum_{i=1}^n i(n+1-i) m_i}\,,
\end{equation} 
for $\omega\in \Omega$, where $m_i= {m}_i(x,y;\omega)$. 
\end{lemma} 
 
\Proof.  Let $x$, $y\in\fX^0$. In view of the preceding results, the proof that topology is independent of the base vertex $y$ proceeds exactly as in the case $n=2$ \cite[Lemma 2.5]{CMS}. 

Now fix $\omega\in \Omega$. Choose $k\in\ZZ_+^n$ such that $k_i\geq d(x,y)$ and
$k_i+m_i(x,y;\omega)\geq d(x,y)$.  Set $z=s_k^x(\omega)=s^y_{k'}(\omega)$,
where $k'={k+m(x,y;\omega)}$.
Lemma \ref{rn5} implies that $\Omega_x^z=\Omega_y^z$. Moreover 
it follows from (\ref{nastyformula}) that
\begin{equation*}
\nu_y(\Omega_x^z)=(|S_{k'}|)^{-1}
    =(q^{\sum_{i=1}^n i(n+1-i) m_i}|S_{k}|)^{-1} 
    = q^{-\sum_{i=1}^n i(n+1-i) m_i}\nu_x(\Omega_x^z)\,.
\end{equation*} 
Since $\{ \Omega_x^z ;\ z\in [x,\omega), d(x,z)\ge d(x,y)\}$ is a basic family of neighbourhoods of $\omega$, the result follows.
\qed 

\begin{remark}
Equation (\ref{RND}) is precisely \cite[Equation (1.6)]{C99}, and its proof is outlined in 
\cite[Section 4]{C99}.
\end{remark}

\subsection{$\tilde A_n$ Groups}\label{sec132}

Let  $\Pi$ be a finite projective geometry of dimension $n$ and order $q$.
If $n>2$ then $\Pi$ is the Desarguesian projective geometry $\Pi(V)$, where $V$ is a
vector space of dimension $n+1$ over a finite field of order $q$. Let $\dim(u)$ denote the dimension of the subspace $u$ of $V$. In the Desarguesian case the points and lines of $\Pi$ are the one- and two- dimensional subspaces of $V$ respectively. We shall extend this notation to the non Desarguesian case, so that an element $u$ of a projective plane $\Pi$ satisfies $\dim u =1$ if it is a point and $\dim u =2$ if it is a line.
Let $\lambda$ be an 
involution of $\Pi$ such that $\dim(\lambda(u))=n+1-\dim(u)\mod(n+1)$.   An $\tilde A_n$ triangle presentation $T$ compatible with
$\lambda$ is defined as follows.  Let $T$ be a set of triples $\{(u,v,w):u,v,w\in\Pi\}$ which
satisfy the following properties.

\begin{enumerate}
\item Given $u$, $v\in\Pi$, then $(u,v,w)\in T$ for some $w\in\Pi$
    if and only if $\lambda(u)$ and $v$ are distinct and incident.
\item If $(u,v,w)\in T$, then $(v,w,u)\in T$.
\item If $(u,v,w_1)\in T$ and $(u,v,w_2)\in T$, then $w_1=w_2$.
\item If $(u,v,w)\in T$, then 
    $(\lambda(w),\lambda(v),\lambda(u))\in T$.
\item If $(u,v,w)\in T$, then $\dim(u)+\dim(v)+\dim(w)\equiv0\mod n+1$.
\end{enumerate}

The group associated with this triangle presentation is given by
\[
\Gamma_T=\left\langle\{a_v\}_{v\in\Pi(x)}\left|
    \begin{matrix}(1) a_{\lambda(v)}=a_v^{-1}
                &\hbox{ for all }v\in\Pi \\
            (2) a_u a_v a_w=1 &\hbox{ for all }
                (u,v,w)\in T\end{matrix}\right.
            \right\rangle\,.
\]

The Cayley graph of $\Gamma_T$, with respect to the generators 
$\{a_u\}_{u\in\Pi}$ 
is the 1-skeleton of an $\tilde A_n$
building $\fX$ and $\Gamma_T$ acts on the vertices of the building in a type
rotating manner.  Conversely any group $\Gamma$ acting on an $\tilde A_n$
building in this way arises as $\Gamma=\Gamma_T$ for some triangle
presentation $T$  \cite[pp 45--46]{C95}. 
Unless otherwise specified,
a generator $a_u$ of $\Gamma$ will be identified with the corresponding element $u\in\Pi$.

\begin{remark}
The type rotating hypothesis in the definition of an $\tilde A_n$ group has been removed  and the appropriate notion of triangle presentation studied in the Ph.D. thesis of T. Svenson \cite{sven}, thereby generalising the results of \cite{C95}.
\end{remark}

For the rest of this article, the $\tilde A_n$ group $\Gamma$ will be assumed to act on $\fX$  by left
translation with $\Gamma$ being identified with the vertex set $\fX^0$.  The identity element $1$ of $\Gamma$ is a preferred vertex of $\fX$ of type $0$, and we write $S_k=S_k(1)$ for $k=(k_1,k_2,\ldots,k_n)\in\ZZ^n_+$. The group $\Gamma$ acts naturally on the boundary $\Omega$.

If $u_1,u_2$ are elements of $\Pi$ we denote by $u_1\vee u_2$ their {\em join}; that is their least upper bound in the lattice of subspaces of $\Pi$. If $\Pi=\Pi(V)$ is Desarguesian then $u_1\vee u_2=\Pi$ means simply that $u_1+u_2=V$. On the other hand, if $\Pi$ is a non Desarguesian plane and $u_1$ is a point and $u_2$ is a line of $\Pi$, then $u_1\vee u_2=\Pi$ means that $u_1$ and $u_2$ are not incident.

By \cite[Lemma 2.2]{C95}, every word in $\Gamma$ can be expressed 
uniquely in {\it normal form}
$$x=u_1u_2\ldots u_l\,,$$
where $\dim (u_i)\leq\dim(u_{i+1})$ and $u_i^{-1}\vee u_{i+1}=\Pi$.  Moreover,
$x\in S_k$, where $k_j=|\{u_i:\dim(u_i)=j\}|$.

Recall from \cite[Proof of Theorem 2.5]{C95} that if $x\in\fX^0$ then the 
projective geometry of neighbours of $x$ is $\{xu:u\in\Pi\}$ and  
$\tau(xu)=\tau(x)+\dim u\mod(n+1)$.  Moreover, $xu$ and $xu'$ are  
adjacent vertices if and only if $u$ and $u'$ are incident in $\Pi$ (that is, $u\subset u'$ or $u'\subset u$). In particular a chamber of $\fX$ 
containing the vertex $x$ has the form 
$$\{x,xu_1,xu_2,\ldots,xu_n\}$$ 
where $\dim u_i=i$ and  
$u_1\subset u_2\subset\cdots\subset u_n$ is a complete flag in $\Pi$. 

For more information on $\tilde A_n$ groups, the reader is referred to 
\cite{C95}.

\bigskip

\section {An Ergodic measure preserving subgroup of the full group.} 
 
The action of an
$\tilde A_n$ group $\Gamma$ on the boundary $\Omega$ of the corresponding
$\tilde A_n$ building,
is measure-theoretically ergodic with respect to each of the  measures
$\nu_y$, ${y\in\fX}$. For the classification of the action it will be necessary to show that the full group $[\Gamma]$ (defined below) contains a countable measure preserving subgroup $K_0\subset[\Gamma]$ which acts ergodically on $\Omega$.

The following two lemmas are straightforward generalisations of 
\cite[Lemma 4.6]{RR} and \cite[Lemma 4.7]{RR} respectively.

\begin{lemma}\label{E1} 
Let $K$ be a group which acts on $\Omega$.  If $K$ acts transitively 
on the collection of sets  
$\{\Omega^x_1:x\in S_k\}$ 
for every $k\in\ZZ_+^n$, then $K$ acts ergodically on $\Omega$. 
\end{lemma} 
 
\Proof.
Observe first that $K$ preserves $\nu_1$ since $\nu_1(\Om_1^x)$ is independent of $x$.
Suppose that $X_0\subseteq\Om$ is a Borel set which is invariant under $K$
and such that $\nu_1(X_0)>0$. It will be shown that 
$\nu_1(\Om\setminus X_0)=0$, thus establishing
the ergodicity of the action.

Define a new measure $\mu$ by $\mu(X)=\nu_1(X\cap X_0)$ for each Borel set
$X\subseteq\Om$. Now, for each $g\in K$,
\begin{eqnarray*}
\mu(gX)&=&\nu_1(gX\cap X_0)\\ 
&=& \nu_1(X\cap g^{-1}X_0) \\
&\leq & \nu_1(X\cap X_0) + \nu_1(X\cap (g^{-1}X_0\setminus X_0)) \\
&=& \nu_1(X\cap X_0) \\
&=& \mu(X).
\end{eqnarray*}
Similarly, $\mu(gX)\leq\mu(g^{-1}gX)=\mu(X)$. 
Therefore $\mu$ is $K$-invariant.

For each $x$, $y\in S_k$ there exists a $g\in K$ such that $g\Om_1^x=\Om_1^y$
by transitivity. Thus $\mu(\Om_1^x) = \mu(\Om_1^y)$. Since
$\Om$ is the union of $|S_k|$ disjoint sets $\Om_1^x$, $y\in S_k$,
each of
which has equal measure, one has that
\[
\mu(\Om_1^x)=\frac{c}{|S_k|}, \text{ for each } x\in S_k,
\]
where $c=\mu(X_0)=\nu_1(X_0)>0$. Thus $\mu(\Om_1^x)=c\nu_1(\Om_1^x)$ for every
vertex $x\in \fX$.

Since the sets $\Om_1^x$ generate the Borel $\s$-algebra, it follows that
$\mu(X)=c\nu_1(X)$ for each Borel set $X$. Therefore
\begin{eqnarray*}
\nu_1(\Om\setminus X_0) &=& c^{-1}\mu(\Om\setminus X_0) \\
&=& c^{-1}\nu_1((\Om\setminus X_0) \cap X_0) = 0 ,
\end{eqnarray*}
thus proving ergodicity.
\qed

\begin{lemma}\label{E2} 
Assume that $K\leq {\rm{Aut}}(\Omega)$ acts transitively on the collection 
of sets  
$\{\Omega^x_1:x\in S_m\}$ 
for every $m\in\ZZ_+^n$.  
Then there is a countable subgroup $K_0$ of $K$ which also acts 
transitively on the collection of sets  
$\{\Omega^x_1:x\in S_m\}$ 
for every $S_m$, $m\in\ZZ_+^n$. 
\end{lemma} 
 
\begin{proof}
For each pair $x$, $y\in S_m$, there exists an element $k\in K$ such that
$k\Om_1^y=\Om_1^x$. Choose one such element $k\in K$ and label it $k_{x,y}$.
Since $S_m$ is finite, there are a finite number of elements
$k_{x,y}\in K$ for
each $S_m$. There are countably many sets $S_m$, so the set
$\{ k_{x,y} : x,y\in S_m , m\in\ZZ_+^n\}$  is countable. Hence the group
\[
K_0=\left< k_{x,y} ; x,y\in S_m , m\in\ZZ_+^n \right>\leq K
\]
is countable and satisfies the required condition.
\end{proof} 

\begin{definition}
Given a group $\Gamma$ acting on a measure space $\Omega$, define the
{\it full group}, $[\Gamma]$, of $\Gamma$ by
$$[\Gamma]=\{T\in\Aut(\Omega);T\omega\in\Gamma\omega\text{ for almost every
            $\omega\in\Omega$}\}\,.$$
The set $[\Gamma]_0$ of measure preserving maps in $[\Gamma]$ is then
given by
$$[\Gamma]_0=\{T\in[\Gamma];\nu_y\circ T=\nu_y,\, y\in\fX^0\}\,.$$
\end{definition}

It will be shown that there is a countable group $K_0$ of  
{\it measure-preserving} automorphisms of $\Omega$ such that 
\begin{enumerate} 
 
\item $K_0$ acts ergodically on $\Omega$. 
 
\item $K_0\leq[\Gamma]$. 
\end{enumerate} 
 
By the Lemmas above and the definition of $[\Gamma]$, it is enough to  
find for each $k\in\ZZ_+^n$ and $x,y\in S_k$, an automorphism  
$g\in\Aut(\Omega)$ such that $g(\Omega_1^x)=\Omega^y_1$ and  
$g\omega\in\Gamma\omega$ for almost  
all $\omega\in\Omega$. 
 
Identify a simplex in $\fX$ with its vertex set, and recall 
from section \ref{sec132} that 
a chamber of $\fX$ containing the vertex $x$ is of the form
$$\{x,xu_1,xu_2,\ldots,xu_n\}\,.$$ 
where $\dim u_i=i$ and $u_1\subset u_2\subset\cdots\subset u_n$ is
a complete flag in $\Pi$. 

\begin{lemma}\label{E3} 
Let $C=\{1,p_1,p_2,\ldots,p_n\}$ be a chamber in $\fX$ with base vertex 
the identity element $1$ of $\Gamma$, where $p_i$ are generators of  
$\Gamma$ and $\dim p_i=i$.  There are $q$ 
chambers $C'=\{x,p_1\ldots,p_n\}$ in $\fX$ meeting $C$ in the face 
$\{p_1,\ldots,p_n\}$. The vertex $x$ opposite $1$ in $C\cup C'$ has 
the normal form $x=p_1u_n$, where $\dim u_n=n$ and $p_1^{-1}\vee u_n=\Pi$. 
Thus $x\in S_{(1,0,\ldots,0,1)}$. 
 
Equivalently, $x=p_nu'_1$, where $\dim u'_1=1$ and $p_n\vee (u'_1)^{-1}=\Pi$. 
\end{lemma}
 
\refstepcounter{picture} 
\begin{figure}[htbp]\label{Ef1} 
{}\hfil 
\font\thinlinefont=cmr5 
\begingroup\makeatletter\ifx\SetFigFont\undefined%
\gdef\SetFigFont#1#2#3#4#5{%
  \reset@font\fontsize{#1}{#2pt}%
  \fontfamily{#3}\fontseries{#4}\fontshape{#5}%
  \selectfont}%
\fi\endgroup%
\centerline{\mbox{\beginpicture 
\setcoordinatesystem units <.25000cm,.250000cm> 
\unitlength=.250000cm 
\linethickness=1pt 
\setplotsymbol ({\makebox(0,0)[l]{\tencirc\symbol{'160}}}) 
\setshadesymbol ({\thinlinefont .}) 
\setlinear 
%
%
\linethickness= 0.500pt 
\setplotsymbol ({\thinlinefont .}) 
\putrule from  3.190 19.689 to  8.270 19.689 
\plot  8.270 19.689  5.730 15.289 / 
\plot  5.730 15.289  3.190 19.689 / 
%
%
\linethickness= 0.500pt 
\setplotsymbol ({\thinlinefont .}) 
\putrule from  3.175 19.685 to  8.255 19.685 
\plot  8.255 19.685  5.715 24.086 / 
\plot  5.715 24.086  3.175 19.685 / 
%
%
\put{\SetFigFont{10}{14.4}{\rmdefault}{\mddefault}{\updefault}{${C}$}} [lB] at  5.25 17.66 
%
%
\put{\SetFigFont{10}{14.4}{\rmdefault}{\mddefault}{\updefault}$C'$} [lB] at  5.25 20.55
%
%
\put{\SetFigFont{10}{14.4}{\rmdefault}{\mddefault}{\updefault}$x$} [lB] at  6.2 23.8 
%
%
\put{\SetFigFont{10}{14.4}{\rmdefault}{\mddefault}{\updefault}$p_1$} [lB] at  8.7 19.685 
%
%
\put{\SetFigFont{10}{14.4}{\rmdefault}{\mddefault}{\updefault}$p_j$} [lB] at  1.7 19.685 
%
%
\put{\SetFigFont{10}{14.4}{\rmdefault}{\mddefault}{\updefault}1} [lB] at  6 14.4 
\linethickness=0pt 
\putrectangle corners at  2.064 24.282 and  8.572 14.707 
\endpicture}} 
\hfill{} 
\caption{} 
\end{figure} 
 
\Proof.  Consider the projective geometry of the neighbours of 
$p_1$. For $2\leq i\leq n$ there exists $u_{i-1}\in\Pi_{i-1}$ such that  
$$p_{i}=p_1u_{i-1}\qquad\hbox{and}\qquad  
        u_{i-1}\subset u_{j-1}\hbox{ for }i\leq j\,.$$ 
Now choose $u_n\in\Pi_n$ such that $u_{n-1}\subset u_n$ and  
$u_n\neq p_1^{-1}$.  There exist $q$ such 
choices for $u_n$.  One then has that for all 
$2\leq i\leq n$, $u_{i-1}\subset u_n$,  
and hence $p_i=p_1u_{i-1}$ is adjacent to $p_1u_n$.   
Thus $C'=\{p_1,p_2,\ldots,p_n,p_1u_n\}$ is a chamber of $\fX$ and 
$p_1u_n$ is the vertex $x$ opposite 1 in $C\cup C'$.   
Clearly $p_1^{-1}\vee u_n=\Pi$, so $x=p_1u_n$ is the  
normal form expressing $x$ as a word of minimal length.   
 
It now follows that $x\in S_{(1,0,\ldots,0,1)}$, and a similar argument 
proves the final statement. 
\qed 
 
\begin{lemma}\label{E8} 
Let $x\in S_k$ and $y\in S_{k'}$, $k$, $k'\in\ZZ_+^n$, where $k=(k_1,\ldots,k_n)$ and
$k'=(k'_1,\ldots,k'_n)$.  Then there 
exists an automorphism $\varphi$ of $\Omega$ such that 
\begin{enumerate} 
\item $\varphi\in[\Gamma]$, the full group of $\Gamma$; 
\item $\varphi$ is almost everywhere a bijection from $\Omega^x_1$ 
onto $\Omega^y_1$; 
\item $\varphi$ is the identity on  
        $\Omega\backslash(\Omega^x_1\cup\Omega^y_1)$. 
\end{enumerate}
Moreover, if $k=k'$ then $\varphi$ is measure preserving. 
\end{lemma} 
 
\Proof.  Let $\delta=e_1+e_n=(1,0,\ldots,0,1)$ and consider the set of all 
vertices $x_1\in S_{k+\delta}$ such that 
 $x\in\conv\{1,x_1\}$.  For such a vertex $x_1$, one has that $x_1\in S_\delta(x)$  
and $\Omega_1^{x_1}=\Omega_x^{x_1}$.  Thus  
$\Omega^x_1$ is a (disjoint) union of sets of the form 
$\Omega^{x_1}_1=\Omega^{x_1}_x$, where  
$x_1\in S_\delta(x)\cap S_{k+\delta}$ is constructed using the procedure  
of Lemma \ref{E3}. 
 
Similarly, $\Omega^y_1$ is a disjoint union of sets of the form 
$\Omega^{y_1}_1=\Omega^{y_1}_y$, where  
$y_1\in S_\delta(y)\cap S_{k'+\delta}$.  Refer to Figure \ref{Ef3}  
below. 

\refstepcounter{picture} 
\begin{figure}[htbp]\label{Ef3} 
\hfil{} 
\font\thinlinefont=cmr5 
\begingroup\makeatletter\ifx\SetFigFont\undefined%
\gdef\SetFigFont#1#2#3#4#5{%
  \reset@font\fontsize{#1}{#2pt}%
  \fontfamily{#3}\fontseries{#4}\fontshape{#5}%
  \selectfont}%
\fi\endgroup%
\centerline{\mbox{\beginpicture 
\setcoordinatesystem units <0.25000cm,0.25000cm> 
\unitlength=0.25000cm 
\linethickness=1pt 
\setplotsymbol ({\makebox(0,0)[l]{\tencirc\symbol{'160}}}) 
\setshadesymbol ({\thinlinefont .}) 
\setlinear 
%
%
\linethickness= 0.500pt 
\setplotsymbol ({\thinlinefont .}) 
\plot  5.080 -3.810  7.279  0.000 / 
\putrule from  7.279  0.000 to  2.881  0.000 
\plot  2.881  0.000  5.080 -3.810 / 
%
%
\linethickness= 0.500pt 
\setplotsymbol ({\thinlinefont .}) 
\plot  5.080 11.430  7.279 15.240 / 
\plot  7.279 15.240  9.478 11.430 / 
\putrule from  9.478 11.430 to  5.080 11.430 
%
%
\linethickness= 0.500pt 
\setplotsymbol ({\thinlinefont .}) 
\plot  5.080  3.810  7.279  0.000 / 
\putrule from  7.279  0.000 to  2.881  0.000 
\plot  2.881  0.000  5.080  3.810 / 
%
%
\linethickness= 0.500pt 
\setplotsymbol ({\thinlinefont .}) 
\plot  5.080 11.430  7.279  7.620 / 
\plot  7.279  7.620  9.478 11.430 / 
\putrule from  9.478 11.430 to  5.080 11.430 
%
%
\linethickness= 0.500pt 
\setplotsymbol ({\thinlinefont .}) 
\plot  5.086  3.797  7.277  7.626 / 
%
%
\linethickness= 0.500pt 
\setplotsymbol ({\thinlinefont .}) 
\plot 22.273 -3.810 24.472  0.000 / 
\putrule from 24.472  0.000 to 20.074  0.000 
\plot 20.074  0.000 22.273 -3.810 / 
%
%
\linethickness= 0.500pt 
\setplotsymbol ({\thinlinefont .}) 
\plot 22.273 11.430 24.472 15.240 / 
\plot 24.472 15.240 26.671 11.430 / 
\putrule from 26.671 11.430 to 22.273 11.430 
%
%
\linethickness= 0.500pt 
\setplotsymbol ({\thinlinefont .}) 
\plot 22.273  3.810 24.472  0.000 / 
\putrule from 24.472  0.000 to 20.074  0.000 
\plot 20.074  0.000 22.273  3.810 / 
%
%
\linethickness= 0.500pt 
\setplotsymbol ({\thinlinefont .}) 
\plot 22.273 11.430 24.472  7.620 / 
\plot 24.472  7.620 26.671 11.430 / 
\putrule from 26.671 11.430 to 22.273 11.430 
%
%
\linethickness= 0.500pt 
\setplotsymbol ({\thinlinefont .}) 
\plot 22.279  3.797 24.470  7.626 / 
%
%
\put{\SetFigFont{8}{14.4}{\rmdefault}{\mddefault}{\updefault}$y_3$} [lB] at 25.082 15 
%
%
\put{\SetFigFont{8}{14.4}{\rmdefault}{\mddefault}{\updefault}$C'_y$} [lB] at 23.7 12.5 
%
%
\put{\SetFigFont{8}{14.4}{\rmdefault}{\mddefault}{\updefault}$C_y$} [lB] at 23.7  9.7 
%
%
\put{\SetFigFont{8}{14.4}{\rmdefault}{\mddefault}{\updefault}$y_1zv_n$} [lB] at 27.2 11.35 
%
%
\put{\SetFigFont{8}{14.4}{\rmdefault}{\mddefault}{\updefault}$y_1zv_np_i$} [lB] at 17.8 11.35 
%
%
\put{\SetFigFont{8}{14.4}{\rmdefault}{\mddefault}{\updefault}$y_1z=y_2$} [lB] at 25.082  7.5
%
%
\put{\SetFigFont{8}{14.4}{\rmdefault}{\mddefault}{\updefault}$y_1$} [lB] at 20.793  3.75 
%
%
\put{\SetFigFont{8}{14.4}{\rmdefault}{\mddefault}{\updefault}$y$} [lB] at 22.893 -4.2 
%
%
\put{\SetFigFont{8}{14.4}{\rmdefault}{\mddefault}{\updefault}$x$} [lB] at  5.7 -4.2 
%
%
\put{\SetFigFont{8}{14.4}{\rmdefault}{\mddefault}{\updefault}$x_1$} [lB] at  3.6  3.75 
%
%
\put{\SetFigFont{8}{14.4}{\rmdefault}{\mddefault}{\updefault}$x_2$} [lB] at  7.9  7.5 
%
%
\put{\SetFigFont{8}{14.4}{\rmdefault}{\mddefault}{\updefault}$x_2v_n$} [lB] at  10 11.35 
%
%
\put{\SetFigFont{8}{14.4}{\rmdefault}{\mddefault}{\updefault}$x_2v_np_i$} [lB] at  1.1 11.35 
%
%
\put{\SetFigFont{8}{14.4}{\rmdefault}{\mddefault}{\updefault}$x_3$} [lB] at  7.9 15 
%
%
\put{\SetFigFont{8}{14.4}{\rmdefault}{\mddefault}{\updefault}$C_x$} [lB] at  6.6  9.7 
%
%
\put{\SetFigFont{8}{14.4}{\rmdefault}{\mddefault}{\updefault}$C'_x$} [lB] at  6.6 12.5 
\linethickness=0pt 
\putrectangle corners at  2.855 15.513 and 27.116 -4.248 
\endpicture}} 
\hfill{} 
\caption{} 
\end{figure}
 
It is therefore enough to show that for every such $x_1$, $y_1$,  
there is a measure preserving bijection  
$\varphi:\Omega^{x_1}_x\to\Omega^{y_1}_y$ which coincides pointwise 
with the action of $\Gamma$  almost  
everywhere on $\Omega^{x_1}_x$.  That is, for almost 
each $\omega\in\Omega^{x_1}_x$, there exists $g\in\Gamma$ such that 
$\varphi\omega=g\omega$. 
 
Fix such $x_1$, $y_1$.  Choose $x_2\in S_{e_n}(x_1)\cap S_{k+\delta+e_n}$.  Also choose  
$v_n\in S_{e_n}$ such that  
$$x_2 v_n\in S_{2e_n}(x_1)\cap S_{k+\delta+2e_n}\,.$$ 
Since $y_1\in S_{k'+\delta}$, it has normal form 
$$y_1=u_1\ldots u_l\qquad\hbox{where }u_l\in S_{e_n}\,.$$ 
 
We now show that there exists $z\in S_{e_n}$ such that  
\begin{equation}\label{Eeq1} 
u^{-1}_l\vee z=\Pi\qquad\hbox{and}\qquad z^{-1}\vee v_n=\Pi\,. 
\end{equation} 
 
To prove the claim, it is necessary to make use of the identification of the  
generators of $\Gamma$ with elements of the finite projective  
space $\Pi$.  Set $\Pi_r=\{x\in\Pi:\dim x=r\}=S_{e_r}$. 
 
Now, $u_l^{-1}\in\Pi_1$, and $v_n\in\Pi_n$.  Therefore 
\begin{description} 
\item[(a)] $|\{z\in\Pi_n:u_l^{-1}\vee z\neq \Pi\}| 
=|\{z\in\Pi_n:u_l^{-1}\subset z\}|=1+q+q^2+\cdots+q^{n-1}$. 
\item[(b)] 
$|\{z\in\Pi_n:z^{-1}\vee v_n\neq \Pi\}| 
=|\{z\in\Pi_n:z^{-1}\subset v_n\}|=1+q+q^2+\cdots+q^{n-1}$. 
\end{description} 
 
Also,  
$$|\Pi_n|=1+q+q^2+\cdots+q^n>2(1+q+\cdots+q^{n-1})\,.$$ 
Hence there exists $z\in\Pi_n$ such that (\ref{Eeq1}) holds. 
 
It follows that the word $y_1zv_n=u_1\ldots u_lzv_n$ is in 
normal form and hence that 
$$y_1zv_n\in S_{k'+\delta+2e_n}\,.$$ 
Moreover, $y_2=y_1z\in S_{k'+\delta+e_n}$. 
 
It will now be shown that the chambers $C_x$, $C_y$ can be constructed 
which lie as indicated in (the two dimensional) Figure \ref{Ef3}. 
By this it is meant, for example, that if $\omega\in\Omega$ and 
$C_x\subset S_{x_2}(\omega)$ then  
$S_{x_2}(\omega)\subset S_{x_1}(\omega)\subset S_x(\omega)$. 
In fact, $C_x=x_2C$ and $C_y=y_1zC$, where $C$ is the chamber 
based at 1, as illustrated in Figure \ref{Ef4} (in two dimensions).

\refstepcounter{picture} 
\begin{figure}[htbp]\label{Ef4} 
\hfill{} 
\font\thinlinefont=cmr5 
\begingroup\makeatletter\ifx\SetFigFont\undefined%
\gdef\SetFigFont#1#2#3#4#5{%
  \reset@font\fontsize{#1}{#2pt}%
  \fontfamily{#3}\fontseries{#4}\fontshape{#5}%
  \selectfont}%
\fi\endgroup%
\centerline{\mbox{\beginpicture 
\setcoordinatesystem units <0.40000cm,0.40000cm> 
\unitlength=0.40000cm 
\linethickness=1pt 
\setplotsymbol ({\makebox(0,0)[l]{\tencirc\symbol{'160}}}) 
\setshadesymbol ({\thinlinefont .}) 
\setlinear 
%
%
\linethickness= 0.500pt 
\setplotsymbol ({\thinlinefont .}) 
\putrule from  3.190 19.689 to  8.270 19.689 
\plot  8.270 19.689  5.730 15.289 / 
\plot  5.730 15.289  3.190 19.689 / 
%
%
\linethickness= 0.500pt 
\setplotsymbol ({\thinlinefont .}) 
\putrule from  3.175 19.685 to  8.255 19.685 
\plot  8.255 19.685  5.715 24.086 / 
\plot  5.715 24.086  3.175 19.685 / 
%
%
\put{\SetFigFont{12}{14.4}{\rmdefault}{\mddefault}{\updefault}$C$} [lB] at  5.321 17.66 
%
%
\put{\SetFigFont{12}{14.4}{\rmdefault}{\mddefault}{\updefault}1} [lB] at  6.1 14.899 
%
%
\put{\SetFigFont{12}{14.4}{\rmdefault}{\mddefault}{\updefault}$v_n$} [lB] at  8.7 19.5 
%
%
\put{\SetFigFont{12}{14.4}{\rmdefault}{\mddefault}{\updefault}$v_np_i$} [lB] at  1.2 19.5 
%
%
\put{\SetFigFont{12}{14.4}{\rmdefault}{\mddefault}{\updefault}$w$} [lB] at  6.151 23.846 
\linethickness=0pt 
\putrectangle corners at  1.064 24.282 and  8.572 14.707 
\endpicture}} 
\hfill{} 
\caption{} 
\end{figure} 

The vertices $x_2v_np_i$, $2\leq i\leq n$, will be constructed from a flag  
$v_n^{-1}\subset p_2\subset p_3\subset\cdots\subset p_n$, where 
$p_i\in\Pi_i$. 
For the following argument, note that if $b\in\Pi_{r-1}$, where $r\geq 2$, then  
$$|\{a\in\Pi_r:a\superset b\}|=1+q+\cdots+q^{n-r+1}\geq 1+q\,.$$ 
 
There are at least $1+q$ elements $p_2\in\Pi_2$ such that  
$p_2\superset v_n^{-1}$.  By reference to both Lemma 2.4 and the proof of  
Proposition 2.7 in \cite{C99}, there is precisely one such $p_2$  
such that $|x_2v_np_2|<|x_2v_n|$. 
(In fact, in that case $x_2v_np_2\in S_{k+\delta+2e_n-e_{n-1}}$.) 
 
Similarly, there is precisely one $p_2\in\Pi_2$, $p_2\superset v_n^{-1}$ 
such that $|y_1zv_np_2|<|y_1zv_n|$. 
 
One can therefore choose $p_2\in\Pi_2$ with $p_2\superset v_n^{-1}$ 
such that $|x_2v_np_2|\geq|x_2v_n|$ and $|y_1zv_np_2|\geq|y_1zv_n|$. 
Moreover, since $v_np_2$ is then adjacent to 1, these inequalities are in 
fact equalities. 
  
This process is now continued. 
There are at least $1+q$ elements $p_3\in\Pi_3$ such that 
$p_3\superset p_2$ and at most two of them satisfy either 
$|x_2v_np_3|<|x_2v_n|$ or $|y_1zv_np_2|<|y_1zv_n|$. 
Thus we may choose $p_3\superset p_2$ such that 
$|x_2v_np_3|=|x_2v_n|$ and $|y_1zv_np_3|=|y_1zv_n|$. 
 
Continue in this way to obtain a flag 
$$v_n^{-1}\subset p_2\subset p_3\subset\cdots\subset p_n$$ 
such that the vertex set of the chamber $C$ is 
$$\{1,v_n,v_np_2,v_np_3,\ldots,v_np_n\}\,.$$ 
Then $C_x=x_2C$ and $C_y=y_2 C$. 
 
Now choose, by Lemma \ref{E3}, a vertex $w$ (one of $q$ possible) 
of a chamber $C_x'$ which meets $C_x$ in the face $C_x\backslash\{1\}$. 
Thus $w\in S_\delta$, and $x_3=x_2w\in S_{k+2\delta+e_n}$. 
Also $y_3=y_2w\in S_{k'+2\delta+e_n}$.  (Recall that, by definition, 
$y_2=y_1z$.)   Moreover, $y_2x_2^{-1}C_x=C_y$. 
 
It has now been shown that  
$$\Omega^{x_3}_{x_2}\subset\Omega^{x_1}_x,\qquad 
    \Omega^{y_3}_{y_2}\subset\Omega^{y_1}_y\qquad\hbox{and}\qquad 
    y_2x_2^{-1}\Omega^{x_3}_{x_2}=\Omega^{y_3}_{y_2}\,.$$ 
Therefore one can define the map $\varphi$ on $\Omega^{x_3}_{x_2}$ by 
$$\varphi \omega=y_2x_2^{-1}\omega\,.$$ 
 
Now recall that $x\in S_k$, $y\in S_{k'}$ and $x_1\in S_{k+\delta}$, 
$y_1\in S_{k'+\delta}$ 
were fixed, and that   
$x_2\in S_{e_n}(x_2)\cap S_{k+\delta+e_n}$ was chosen.  
The set $\Omega^{x_1}_x$  
is a disjoint union of sets of the form $\Omega^{x_3}_{x_2}$ where 
$x_3\in S_\delta(x_2)$.  Let $K$ denote the number of such sets.   
This number is independent of the choice of $x$, $x_1$ and $k$ by  
\cite[Lemma 2.4]{C99}, (or by the fact that $\Gamma$ acts 
simply transitively on $\fX^0$). 
 
The definition $\varphi \omega=y_2x_2^{-1}\omega$ in the above choice of  
$\Omega^{x_3}_{x_2}$ therefore leaves $\varphi$ undefined on a proportion  
$(1-\frac{1}{K})$ of $\Omega_x^{x_1}$.  However, where $\varphi$ is defined  
it coincides with the action of an element of $\Gamma$, namely  
$y_2x_2^{-1}$. 
 
Now repeat the process on each of the $K-1$ sets of the form  
$\Omega^{x_3}_{x_2}$ where $\varphi$ has not been defined.  As before, 
$\varphi$ can be defined except on a proportion $(1-\frac{1}{K})$ of each such set, 
and $\varphi$ can therefore be defined everywhere except on a proportion 
$(1-\frac{1}{K})^2$ of the original set $\Omega_x^{x_1}$. 
 
Continuing in this manner, one sees that at the $n^{th}$ step, $\varphi$  
has been defined everywhere except on a proportion $(1-\frac{1}{K})^n$  
of $\Omega^{x_1}_x$. 
 
Since $(1-\frac{1}{K})^n\to0$ as $n\to\infty$, $\varphi$ is defined  
almost everywhere on $\Omega^{x_1}_x$ and satisfies the required  
properties. If $k=k'$ then it is clear from the construction that $\varphi$ is measure preserving.
\qed

\begin{remark} 
This result extends \cite[Proposition 4.9]{RR}.  Moreover,   
for $n=2$ this proof deals with the case $q=2$, which was 
left open in \cite{RR}.  Thus the hypothesis $q\geq3$ in the main result  
Theorem 4.10 of \cite{RR} is not in fact necessary. 
\end{remark}
 
We can now prove the following. 
 
\begin{proposition}\label{E5} 
There exists a countable subgroup $K_0$ of $[\Gamma]$ such that 
\begin{enumerate} 
\item $K_0$ is measure preserving; 
\item $K_0$ acts ergodically on $\Omega$. 
\end{enumerate} 
\end{proposition} 
 
\Proof. It suffices to take the group generated by the  
automorphisms of the form $\varphi$ defined in Lemma \ref{E8}, with $k=k'$ 
then use Lemma \ref{E2} to extract a countable subgroup 
$K_0$.  Finally, Lemma \ref{E1} shows that the action of $K_0$ 
is ergodic.\qed 
 
\begin{lemma}\label{E6} 
Let  
$\Gamma$ be a countable group acting on a measure space  
$(\Omega,\mu)$.  Suppose that the action of the full group $[\Gamma]$ 
is ergodic.  Then so is the action of $\Gamma$. 
\end{lemma} 
 
\Proof.  Let $S$ be a measurable subset of $\Omega$ such that  
$\mu(gS\backslash S)=0$ for all $g\in\Gamma$. 
Let $k\in [\Gamma]$.  It will be shown that $\mu(kS\backslash S)=0$.   
For each $g\in\Gamma$,  
let  
$$S_g=\{\omega\in S:k\omega=g\omega\}\,,$$ 
which is a measurable subset of $S$.  Since $k\in[\Gamma]$ it follows that 
$$S=S_0\cup\bigcup_{g\in\Gamma}S_g\,,$$ 
where $S_0$ has measure zero. 
 
Then  
\begin{eqnarray*} 
\mu(kS\backslash S)&\leq&\sum_{g\in \Gamma}\mu(kS_g\backslash S)\\ 
&=&\sum_{g\in\Gamma}\mu(gS_g\backslash S)\\ 
&\leq&\sum_{g\in\Gamma}\mu(gS\backslash S)=0\,. 
\end{eqnarray*} 
Thus $\mu(kS\backslash S)=0$ for all $k\in[\Gamma]$.  Since the 
action of $[\Gamma]$ is ergodic it follows that $S$ is 
either null or conull with respect to the measure $\mu$. 
Therefore the action of $\Gamma$ is ergodic.\qed

\begin{corollary}\label{E7} 
The action of $\Gamma$ on $\Omega$ is ergodic. 
\end{corollary} 
 
\Proof. This follows from Proposition \ref{E5} and 
Lemma \ref{E6}.\qed

\bigskip
\section{Classification of the action of $\Gamma$ on $\Omega$}\label{classification}

Having shown that the action of an $\tilde A_n$ group on its boundary
$\Omega$ is ergodic, we now show that it is type $III_{\lambda}$, where 
$0\leq\lambda\leq 1$, and  the value of $\lambda$ depends 
on the {\em ratio set}.

\begin{definition}\label{rn1} 
Let $\Gamma$ be a countable group of automorphisms of 
the measure space $(\Omega,\nu)$. 
Following Krieger, define the {\it ratio set} $r(\Gamma)$ 
to be the set of $\lambda\in[0,\infty)$ such that 
for every $\epsilon>0$  
and Borel set $\cE$ 
with $\nu(\cE)>0$, there exists a $g\in\Gamma$ and a  Borel 
set $\cF$ such that $\nu(\cF)>0$, 
$\cF\cup g\cF\subset\cE$ and  
$$\left|\frac{d\nu\circ g}{d\nu}(\omega)-\lambda\right|<\epsilon\,$$ 
for all $\omega\in\cF$. 
\end{definition} 
 
\begin{remark}\label{rn7} 
The ratio set $r(\Gamma)$ depends only on the quasi-equivalence 
class of the measure $\nu$ \cite[section I-3, Lemma 14]{HO}. It also depends only 
on the full group in the sense that 
$$[H]=[G]\implies r(H)=r(G)\,.$$ 
\end{remark} 
 
\begin{proposition}\label{rn8} 
Let $\fX$ be a locally finite, thick $\tilde A_n$ building of order
$q$, and let $\Gamma$ be a countable group of type rotating automorphisms of $\fX$. 
Fix a vertex $O\in\fX^0$ of type $0$, and suppose  
that for each $0\leq i\leq n$ there exists an element 
$g_i\in\Gamma$ such that  
$d(g_iO,O)=1$ and $g_iO$ is a vertex of type $i$.  Also, 
suppose that there exists a countable 
subgroup $K$ of $[\Gamma]_0$ whose action on $\Omega$ is ergodic.  
Then  
$$r(\Gamma)=\left\{\begin{matrix} 
        \{q^{n}:n\in\ZZ\}\cup\{0\}&\hbox{for $n$ odd}\\ 
        \{q^{2n}:n\in\ZZ\}\cup\{0\}&\hbox{for $n$ even} 
        \end{matrix}\right.\,.$$ 
\end{proposition} 
 
\Proof.
By Remark \ref{rn7}, it is sufficient to prove the 
statement for some group $H$ 
such that $[H]=[\Gamma]$.  In particular, since  
$[\Gamma]=[\langle \Gamma,K\rangle]$ for any subgroup $K$ of  
$[\Gamma]_0$, we may assume without loss of generality that 
$K\leq \Gamma$. 

Let $\nu=\nu_O$.  
For  $g_i\in\Gamma$ as in the statement of the lemma,
let  $x=g_iO$, and note that
$\nu_{x}=\nu\circ g_i^{-1}$. 
If ${m}(O,x;\omega)=(m_1,m_2,\ldots,m_n)$ then by Lemma \ref{rn6},  
$$\frac{d\nu\circ g^{-1}_i}{d\nu}(\omega)
    =\frac{d\nu_{x}}{d\nu}(\omega)
    =q^{-\sum_{i=1}^n i(n+1-i) m_i}\,.$$ 

Then for $\omega\in\Omega_O^x$, one has that 
${m}(O,x;\omega)=(0,\ldots,0,-1,0,\ldots,0)$, where the $-1$ is in the
$i^{th}$ place.  Thus 
\begin{equation}\label{rne1}
\frac{d\nu_x}{d\nu}(\omega)=q^{i(n+1-i)}\qquad
                \hbox{for $\omega\in\Omega_O^x$.}
\end{equation}
Let $\cE\subset\Omega$ be a Borel set with $\nu(\cE)>0$.  Then by the
ergodicity of $K$, there exist $k_1$, $k_2\in K$ such that the set
$$\cF=\{\omega\in\cE:k_1\omega\in\Omega_O^x\hbox{ and }
            k_2g_i^{-1}k_1\omega\in\cE\}$$
has positive measure.

Next, let $t=k_2g_i^{-1}k_1\in \G$.  By construction, $\cF\cup t\cF\subset\cE$.
Moreover, since $K$ is measure preserving,
$$\frac{d\nu\circ t}{d\nu}(\omega)
    =\frac{d\nu\circ g^{-1}_i}{d\nu}(k_1\omega)
    =\frac{d\nu_{x_i}}{d\nu}(k_1\omega)=q^{i(n+1-i)}\qquad\hbox{for all
    $\omega\in\cF$}$$
by (\ref{rne1}), and since $k_i\omega\in\Omega_O^x$.  Hence 
$q^{i(n+1-i)}\in r(\Gamma)$ for $1\leq i\leq n$.

Since the action of $\Gamma$ on $\Omega$ is ergodic, $r(\Gamma)-\{0\}$ forms  
a group.  It is now possible to determine the generator
of $r(\Gamma)-\{0\}$.

Suppose that $n$ is odd.  Then for $i\in\{1,2\}$, one has that $q^n$, 
$q^{2(n-1)}\in r(\Gamma)$.  As $n$, $2(n-1)$ are coprime for $n$ odd,
and as $r(\Gamma)-\{0\}$ forms a group, it follows that $q\in r(\Gamma)$,
and hence 
$$r(\Gamma)=\{q^n:n\in\ZZ\}\qquad\hbox{for $n$ odd.}$$

Suppose that $n$ is even.  As before, $q^n$, $q^{2(n-1)}\in r(\Gamma)$.  Moreover,
as highest common factor of $n$, $2(n-1)$ is 2 for $n$ even, and as
$r(\Gamma)$ forms a group, it follows that $q^2\in r(\Gamma)$.
Finally, as $i(n+1-i)$ is even for all $i$ if $n$ is even, it follows
that for $g\in\Gamma$, and $x=g^{-1}O$, 
$$\frac{d\nu\circ g}{d\nu}=\frac{d\nu_x}{d\nu}=q^{-\sum i(n+1-i)m_i}\in
    \{q^{2n}:n\in\ZZ\}\,.$$
Thus 
$$r(\Gamma)=\{q^{2n}:n\in\ZZ\}\qquad\hbox{for $n$ even.}$$
\qed

\begin{proposition}\label{P2}
Let $\Gamma$ be an $\tilde A_n$ group.  Then  
the action of $\Gamma$ on $\Omega$ is amenable.
\end{proposition}

\Proof.
This is a straightforward generalization of the case $n=2$, proved in \cite[Proposition 3.14]{RR}.
\qed

\subsection{Proof of Theorem \ref{main}.}\label{proofmain}
This follows from Proposition \ref{E5}, Corollary \ref{E7}, Proposition \ref{rn8} and
Theorem \ref{P2}.\qed

\subsection{Proof of Theorem \ref{main*}.}
\label{classification_subsection} 

The proof of Theorem \ref{main*} is now easy.
Let $\fX$ be the affine building
of $G=\PGL(n+1,\QQ_p)$. 
By \cite[Proposition VI.9F]{Brown}, the  
boundary $\Omega$ of $\fX$ is isomorphic to $\PGL(n+1,\QQ_p)/B$ as a  
topological $G$-space.  The measure $\mu$ on $G/B$ is (up to equivalence) the 
natural quasi-invariant Borel measure on $G/B$. 

The vertex set $\fX^0$ of $\fX$ is identified with
$\PGL(n+1,\QQ_p)/\PGL(n+1,\ZZ_p)$, where $\ZZ_p$ is the ring of $p$-adic integers.  
Let $O=\PGL(n+1,\ZZ_p)\in \fX^0$.
It follows from \cite[Proposition 3.1]{Steg} that $\PGL(n+1,\ZZ_p)$ 
acts transitively on each set $S_k(O)$.  Since the vertex set 
$\fX^0=\PGL(n+1,\QQ_p)/\PGL(n+1,\ZZ_p)$ is a discrete space 
and $\ZZ$ is dense in $\ZZ_p$ it follows that $\PGL(n+1,\ZZ)$ 
also acts transitively on $S_k(O)$ for each $k\in\ZZ_+^n$.
Moreover $\PGL(n+1,\ZZ)$ stabilizes $O$. 
Therefore $\PGL(n+1,\ZZ)$ acts  
ergodically on $\Omega$ by Lemma \ref{E1}.  The 
group $\PGL(n+1,\QQ)$ also acts transitively 
on $\fX^0$.  By our previous computation of Radon-Nikodym derivatives 
and the argument of \cite[Proposition 4.4]{RR}   
the type of the action is as stated.

Note that in this argument there is no need to consider the full group, since $\PGL(n+1,\ZZ)$ is already a measure preserving ergodic subgroup of $\PGL(n+1,\QQ)$.
Thus the proof is considerably simpler than the proof of Theorem \ref{main}.  
 
\bigskip

\section{Freeness of the action on the boundary}\label{freeness} 

A simple modification of \cite[Proposition 4.12]{RR} shows that if $\FF$ is a (possibly non commutative) local field  then
the action of $\PGL(n+1,\FF)$ on its Furstenberg boundary $\Omega$ 
is measure-theoretically free. 
If $\fX$ is a thick, locally finite affine building of type  
$\tilde A_n$, where $n\geq 3$, then $\fX$ is the building of 
$\PGL(n+1,\FF)$ for some such local field \cite[p137]{Ronan}. 
All known type-rotating $\tilde A_n$ groups, with $n\ge3$,  
embed in $\PGL(n+1,\FF)$ and act upon the building 
of $\PGL(n+1,\FF)$ in the canonical way.  For such groups 
$\Gamma$, the action on $\Omega$ is therefore measure 
theoretically free.

The case of $\tilde A_2$ groups is more interesting because the associated $\tilde A_2$ building may not be the affine building of a linear group. In fact this is the case for the buildings of many of the groups constructed in \cite{CMSZ}.
The boundary action of $\Gamma$ is nevertheless free.

\begin{proposition}\label{free2} 
Let $\fX$ be an $\tilde A_2$ building, and let $\Gamma$ be an 
$\tilde A_2$ group.  Then the action of $\Gamma$ on $\Omega$ is  
measure-theoretically free \cite[c.f. Proposition 3.10]{RR}. 
\end{proposition}

This result was stated in
\cite[Proposition 3.10]{RR}. The proof given there contains a gap, because the proof of \cite[Lemma 3.7, Case 2]{RR} is not complete. Our purpose is to fill this gap (Lemma \ref{free3} below). A similar argument applies  to the $\tilde A_n$ case, where $n\ge 3$. In view of the comments at the beginning of this section, we have confined ourselves to the $\tilde A_2$ case. 

Let $O$ be a fixed vertex in $\fX$ and let $n\in\ZZ_+$.  
Any sector $S$ based at 
$O$ has at its base a triangle $T_n$ with apex $O$ consisting of all 
vertices $v$ in $S$ such that $d(v,O)\leq n$.  (See Figure \ref{freefig1} 
below). 
 
\refstepcounter{picture}
\begin{figure}[htbp]\label{freefig1}
\hfil{}
\font\thinlinefont=cmr5
\begingroup\makeatletter\ifx\SetFigFont\undefined%
\gdef\SetFigFont#1#2#3#4#5{%
  \reset@font\fontsize{#1}{#2pt}%
  \fontfamily{#3}\fontseries{#4}\fontshape{#5}%
  \selectfont}%
\fi\endgroup%
\centerline{\mbox{\beginpicture
\setcoordinatesystem units <.600000cm,.600000cm>
\unitlength=.600000cm
\linethickness=1pt
\setplotsymbol ({\makebox(0,0)[l]{\tencirc\symbol{'160}}})
\setshadesymbol ({\thinlinefont .})
\setlinear
%
%
\linethickness= 0.500pt
\setplotsymbol ({\thinlinefont .})
\plot  8.890 12.700  9.991 14.605 /
\putrule from  9.991 14.605 to  7.789 14.605
\plot  7.789 14.605  8.890 12.700 /
%
%
\linethickness= 0.500pt
\setplotsymbol ({\thinlinefont .})
\plot  8.890 16.510  9.991 14.605 /
\putrule from  9.991 14.605 to  7.789 14.605
\plot  7.789 14.605  8.890 16.510 /
%
%
\linethickness= 0.500pt
\setplotsymbol ({\thinlinefont .})
\putrule from  8.890 16.510 to 11.089 16.510
\plot 11.089 16.510  9.989 14.605 /
\plot  9.989 14.605  8.890 16.510 /
%
%
\linethickness= 0.500pt
\setplotsymbol ({\thinlinefont .})
\plot  8.890 16.510  9.991 14.605 /
\plot  9.991 14.605  7.789 14.607 /
\plot  7.789 14.607  8.890 16.510 /
%
%
\linethickness= 0.500pt
\setplotsymbol ({\thinlinefont .})
\plot  8.890 16.510  7.791 14.605 /
\plot  7.791 14.605  6.691 16.512 /
\plot  6.691 16.512  8.890 16.510 /
%
%
\linethickness= 0.500pt
\setplotsymbol ({\thinlinefont .})
\plot  8.890 16.514  7.791 18.419 /
\plot  7.791 18.419  6.691 16.512 /
\plot  6.691 16.512  8.890 16.514 /
%
%
\linethickness= 0.500pt
\setplotsymbol ({\thinlinefont .})
\putrule from  8.890 16.510 to 11.089 16.510
\plot 11.089 16.510  9.989 18.415 /
\plot  9.989 18.415  8.890 16.510 /
%
%
\linethickness= 0.500pt
\setplotsymbol ({\thinlinefont .})
\plot  8.890 16.510  9.991 18.415 /
\plot  9.991 18.415  7.789 18.413 /
\plot  7.789 18.413  8.890 16.510 /
%
%
\linethickness= 0.500pt
\setplotsymbol ({\thinlinefont .})
\plot  9.989 18.415  8.888 20.320 /
\plot  8.888 20.320  7.789 18.413 /
\plot  7.789 18.413  9.989 18.415 /
%
%
\linethickness= 0.500pt
\setplotsymbol ({\thinlinefont .})
\putrule from  8.888 20.318 to  6.687 20.318
\plot  6.687 20.318  7.789 18.413 /
\plot  7.789 18.413  8.888 20.318 /
%
%
\linethickness= 0.500pt
\setplotsymbol ({\thinlinefont .})
\plot  6.689 20.316  5.588 18.411 /
\plot  5.588 18.411  7.789 18.413 /
\plot  7.789 18.413  6.689 20.316 /
%
%
\linethickness= 0.500pt
\setplotsymbol ({\thinlinefont .})
\plot  5.590 18.411  6.691 16.506 /
\plot  6.691 16.506  7.789 18.413 /
\plot  7.789 18.413  5.590 18.411 /
%
%
\linethickness= 0.500pt
\setplotsymbol ({\thinlinefont .})
\putrule from  9.989 18.415 to  7.789 18.415
\plot  7.789 18.415  8.890 16.510 /
\plot  8.890 16.510  9.989 18.415 /
%
%
\linethickness= 0.500pt
\setplotsymbol ({\thinlinefont .})
\plot  9.989 18.415  8.888 16.510 /
\plot  8.888 16.510 11.089 16.512 /
\plot 11.089 16.512  9.989 18.415 /
%
%
\linethickness= 0.500pt
\setplotsymbol ({\thinlinefont .})
\plot  9.989 18.415 11.087 16.510 /
\plot 11.087 16.510 12.188 18.417 /
\plot 12.188 18.417  9.989 18.415 /
%
%
\linethickness= 0.500pt
\setplotsymbol ({\thinlinefont .})
\plot  9.989 18.415 12.188 18.413 /
\plot 12.188 18.413 11.087 20.320 /
\plot 11.087 20.320  9.989 18.415 /
%
%
\linethickness= 0.500pt
\setplotsymbol ({\thinlinefont .})
\plot  9.989 18.415 11.089 20.318 /
\putrule from 11.089 20.318 to  8.888 20.318
\plot  8.888 20.318  9.989 18.415 /
%
%
\linethickness= 0.500pt
\setplotsymbol ({\thinlinefont .})
\plot 12.188 18.413 11.089 20.318 /
\plot 11.089 20.318  9.989 18.411 /
\plot  9.989 18.411 12.188 18.413 /
%
%
\linethickness= 0.500pt
\setplotsymbol ({\thinlinefont .})
\plot 13.288 20.316 11.089 20.318 /
\plot 11.089 20.318 12.190 18.411 /
\plot 12.190 18.411 13.288 20.316 /
%
%
\linethickness= 0.500pt
\setplotsymbol ({\thinlinefont .})
\plot 12.190 22.221 11.089 20.318 /
\putrule from 11.089 20.318 to 13.291 20.318
\plot 13.291 20.318 12.190 22.221 /
%
%
\linethickness= 0.500pt
\setplotsymbol ({\thinlinefont .})
\plot  9.991 22.223 11.089 20.318 /
\plot 11.089 20.318 12.190 22.225 /
\plot 12.190 22.225  9.991 22.223 /
%
%
\linethickness= 0.500pt
\setplotsymbol ({\thinlinefont .})
\plot  8.890 20.320 11.089 20.318 /
\plot 11.089 20.318  9.989 22.225 /
\plot  9.989 22.225  8.890 20.320 /
%
%
\linethickness= 0.500pt
\setplotsymbol ({\thinlinefont .})
\plot 11.091 20.320 13.291 20.318 /
\plot 13.291 20.318 12.190 22.225 /
\plot 12.190 22.225 11.091 20.320 /
%
%
\linethickness= 0.500pt
\setplotsymbol ({\thinlinefont .})
\plot 13.291 20.322 14.391 22.225 /
\putrule from 14.391 22.225 to 12.190 22.225
\plot 12.190 22.225 13.291 20.322 /
%
%
\linethickness= 0.500pt
\setplotsymbol ({\thinlinefont .})
\plot  8.890 20.320  9.991 22.223 /
\putrule from  9.991 22.223 to  7.789 22.223
\plot  7.789 22.223  8.890 20.320 /
%
%
\linethickness= 0.500pt
\setplotsymbol ({\thinlinefont .})
\plot  8.890 20.320  7.791 22.225 /
\plot  7.791 22.225  6.691 20.318 /
\plot  6.691 20.318  8.890 20.320 /
%
%
\linethickness= 0.500pt
\setplotsymbol ({\thinlinefont .})
\plot  7.789 22.223  5.590 22.225 /
\plot  5.590 22.225  6.691 20.318 /
\plot  6.691 20.318  7.789 22.223 /
%
%
\linethickness= 0.500pt
\setplotsymbol ({\thinlinefont .})
\plot  5.590 22.221  4.489 20.318 /
\putrule from  4.489 20.318 to  6.691 20.318
\plot  6.691 20.318  5.590 22.221 /
%
%
\linethickness= 0.500pt
\setplotsymbol ({\thinlinefont .})
\plot  4.492 20.316  5.590 18.411 /
\plot  5.590 18.411  6.691 20.318 /
\plot  6.691 20.318  4.492 20.316 /
%
%
\linethickness= 0.500pt
\setplotsymbol ({\thinlinefont .})
\plot  4.492 20.316  6.691 20.314 /
\plot  6.691 20.314  5.590 22.221 /
\plot  5.590 22.221  4.492 20.316 /
%
%
\linethickness= 0.500pt
\setplotsymbol ({\thinlinefont .})
\plot  4.492 20.316  5.592 22.219 /
\putrule from  5.592 22.219 to  3.391 22.219
\plot  3.391 22.219  4.492 20.316 /
%
%
\linethickness= 0.500pt
\setplotsymbol ({\thinlinefont .})
\plot 10.319 22.225  6.985 16.510 /
%
%
\linethickness= 0.500pt
\setplotsymbol ({\thinlinefont .})
\plot 10.478 22.225  7.144 16.510 /
%
%
\linethickness= 0.500pt
\setplotsymbol ({\thinlinefont .})
\plot 10.160 22.225  6.826 16.510 /
%
%
\linethickness= 0.500pt
\setplotsymbol ({\thinlinefont .})
\plot 10.636 22.225  7.303 16.510 /
%
%
\linethickness= 0.500pt
\setplotsymbol ({\thinlinefont .})
\plot 10.795 22.225  7.461 16.510 /
%
%
\linethickness= 0.500pt
\setplotsymbol ({\thinlinefont .})
\plot 10.954 22.225  7.620 16.510 /
%
%
\linethickness= 0.500pt
\setplotsymbol ({\thinlinefont .})
\plot  7.779 16.510 11.113 22.225 /
%
%
\linethickness= 0.500pt
\setplotsymbol ({\thinlinefont .})
\plot 11.271 22.225  7.938 16.510 /
%
%
\linethickness= 0.500pt
\setplotsymbol ({\thinlinefont .})
\plot  8.096 16.510 11.430 22.225 /
%
%
\linethickness= 0.500pt
\setplotsymbol ({\thinlinefont .})
\plot 11.589 22.225  8.255 16.510 /
%
%
\linethickness= 0.500pt
\setplotsymbol ({\thinlinefont .})
\plot  8.414 16.510 11.748 22.225 /
%
%
\linethickness= 0.500pt
\setplotsymbol ({\thinlinefont .})
\plot 11.906 22.225  8.572 16.510 /
%
%
\linethickness= 0.500pt
\setplotsymbol ({\thinlinefont .})
\plot  8.731 16.510 12.065 22.225 /
%
\put{\SetFigFont{12}{14.4}{\rmdefault}{\mddefault}{\updefault}$O$} [lB] at 8.763 12 
\put{\SetFigFont{9}{14.4}{\rmdefault}{\mddefault}{\updefault}$C$} [lB] at  8.763 13.748
%
%
\put{\SetFigFont{9}{14.4}{\rmdefault}{\mddefault}{\updefault}1} [lB] at  8.731 15.304
%
%
\put{\SetFigFont{9}{14.4}{\rmdefault}{\mddefault}{\updefault}2} [lB] at  9.938 15.716
%
%
\put{\SetFigFont{9}{14.4}{\rmdefault}{\mddefault}{\updefault}3} [lB] at  9.970 17.113
%
%
\put{\SetFigFont{9}{14.4}{\rmdefault}{\mddefault}{\updefault}4} [lB] at 11.049 17.590
%
%
\put{\SetFigFont{9}{14.4}{\rmdefault}{\mddefault}{\updefault}.} [lB] at 11.049 19.018
%
%
\put{\SetFigFont{9}{14.4}{\rmdefault}{\mddefault}{\updefault}.} [lB] at 12.065 19.590
%
%
\put{\SetFigFont{9}{14.4}{\rmdefault}{\mddefault}{\updefault}.} [lB] at 12.097 20.987
%
%
\put{\SetFigFont{9}{14.4}{\rmdefault}{\mddefault}{\updefault}2$n$-2} [lB] at 12.8 21.431
%
%
\put{\SetFigFont{9}{14.4}{\rmdefault}{\mddefault}{\updefault}2$n$-1} [lB] at  7.24 15.716
%
%
\put{\SetFigFont{9}{14.4}{\rmdefault}{\mddefault}{\updefault}$2n$} [lB] at  6.36 17.494
%
%
\put{\SetFigFont{9}{14.4}{\rmdefault}{\mddefault}{\updefault}3$n$-3} [lB] at  3.96 21.336
\linethickness=0pt
\putrectangle corners at  3.365 22.250 and 14.417 12.675
\endpicture}}
\hfil{}
\caption{}
\end{figure}

Consider the set $\fS_n$ of all such triangles $T_n$ in $\fX$.

\begin{lemma}\label{free1} 
Let $\fS_n$ be as above.  Then 
$$|\fS_n|=(q^2+q+1)(q+1)q^{3n-3}\,.$$ 
\end{lemma} 
 
\Proof.  First note that there are $(q^2+q+1)(q+1)$ choices 
for the base chamber $C$. 
 
The reason for this is that one edge of $C$ containing $O$  
is determined by the number of points $P$ in a finite projective  
plane $\Pi$ of order $q$.  There are $q^2+q+1$ such points $P$. 
See Figure \ref{freefig2} below. 
 
\refstepcounter{picture}
\begin{figure}[htbp]\label{freefig2}
\font\thinlinefont=cmr5
\begingroup\makeatletter\ifx\SetFigFont\undefined%
\gdef\SetFigFont#1#2#3#4#5{%
  \reset@font\fontsize{#1}{#2pt}%
  \fontfamily{#3}\fontseries{#4}\fontshape{#5}%
  \selectfont}%
\fi\endgroup%
\centerline{\mbox{\beginpicture
\setcoordinatesystem units <.2500000cm,.2500000cm>
\unitlength=.2500000cm
\linethickness=1pt
\setplotsymbol ({\makebox(0,0)[l]{\tencirc\symbol{'160}}})
\setshadesymbol ({\thinlinefont .})
\setlinear
%
%
\linethickness= 0.500pt
\setplotsymbol ({\thinlinefont .})
\plot 10.160 16.510 14.558 24.130 /
\putrule from 14.558 24.130 to  5.762 24.130
\plot  5.762 24.130 10.160 16.510 /
%
%
\put{\SetFigFont{12}{14.4}{\rmdefault}{\mddefault}{\updefault}$O$} [lB] at 9.6 14.8
%
%
\put{\SetFigFont{12}{14.4}{\rmdefault}{\mddefault}{\updefault}$C$} [lB] at  9.6 20.923
%
%
\put{\SetFigFont{12}{14.4}{\rmdefault}{\mddefault}{\updefault}$P$} [lB] at 14.891 24.130
%
%
\put{\SetFigFont{12}{14.4}{\rmdefault}{\mddefault}{\updefault}$L$} [lB] at  4.3 24.13
\linethickness=0pt
\putrectangle corners at  5.144 24.390 and 14.891 15.723
\endpicture}}
\hfil{}
\caption{}
\end{figure}

Having chosen the point $P$, there are precisely $q+1$ possible  
lines $L$ in $\Pi$ which are incident with $P$.  There are therefore 
$(q^2+q+1)(q+1)$ choices of $C$. 
 
Having chosen $C$, there are $q$ choices for each of the chambers 
labeled $1,2,3,4,\ldots,(2n-2)$ in the figure. 
Then choose the chamber labeled $(2n-1)$ (q choices) in Figure 
\ref{freefig1}.  This choice 
then determines the whole shaded region in the figure (which 
is contained in the convex hull of the chambers already chosen, 
and hence is uniquely determined).  Now choose the chamber labeled $2n$  
($q$ choices) and continue the process until finally chamber  
$3n-3$ is chosen.  This determines the triangle completely and  
there are $(q^2+q+1)(q+1)q^{3n-3}$ possibilities altogether. 
\qed 
 
This demonstrates that for each positive integer $n$, the boundary 
$\Omega$ of $\fX$ is partitioned into $(q^2+q+1)(q+1)q^{3n-3}$ 
sets $\{\Omega_T:T\in\fS_n\}$, where 
$$\Omega_T=\{\omega\in\Omega:T\subset[O,\omega)\}\,.$$ 
Moreover each of these sets has the same measure \cite{CMS}. 
 
The proof of \cite[Lemma 3.7, case 2.]{RR} can now be completed.

\begin{lemma}\label{free3} 
Let $W$ be a wall of $\fX$ and let $\Sigma$ denote the set 
of boundary points $\omega\in\Omega$ such that for some vertex 
$v$, the sector $[v,\omega)$ lies in an apartment containing 
$W$.  Then  
$$\nu_O(\Sigma)=0\,.$$ 
\end{lemma} 
 
\Proof.  By translating to a parallel sector, one can assume that  
$v=O$.  Also, $W$ is the union of two sector panels, 
which will be denoted by $[O,W^+)$, $[O,W^-)$. 
 
Given $n\in\ZZ_+$, let $\fS_n^+$, $\fS_n^-$, $\fS_n^\bot$ denote the subsets of $\fS_n$
consisting of 
triangles $T^+_n$, $T_n^-$, $T_n^\bot$ respectively, lying in some  
apartment containing $W$ as illustrated below in Figure \ref{freefig3}. 

\refstepcounter{picture}
\begin{figure}[htbp]\label{freefig3}
\font\thinlinefont=cmr5
\begingroup\makeatletter\ifx\SetFigFont\undefined%
\gdef\SetFigFont#1#2#3#4#5{%
  \reset@font\fontsize{#1}{#2pt}%
  \fontfamily{#3}\fontseries{#4}\fontshape{#5}%
  \selectfont}%
\fi\endgroup%
\centerline{\mbox{\beginpicture
\setcoordinatesystem units <.50000cm,.50000cm>
\unitlength=.50000cm
\linethickness=1pt
\setplotsymbol ({\makebox(0,0)[l]{\tencirc\symbol{'160}}})
\setshadesymbol ({\thinlinefont .})
\setlinear
%
%
\linethickness= 0.500pt
\setplotsymbol ({\thinlinefont .})
\putrule from 10.160 16.510 to  5.080 16.510
%
%
\linethickness= 0.500pt
\setplotsymbol ({\thinlinefont .})
\plot 10.160 16.510  7.620 20.908 /
%
%
\linethickness= 0.500pt
\setplotsymbol ({\thinlinefont .})
\plot 10.160 16.510 12.700 20.908 /
%
%
\linethickness= 0.500pt
\setplotsymbol ({\thinlinefont .})
\putrule from 10.160 16.510 to 15.240 16.510
%
%
\linethickness= 0.500pt
\setplotsymbol ({\thinlinefont .})
\putrule from  5.080 16.510 to  2.540 16.510
%
%
\linethickness= 0.500pt
\setplotsymbol ({\thinlinefont .})
\putrule from 15.240 16.510 to 17.780 16.510
%
%
\linethickness= 0.500pt
\setplotsymbol ({\thinlinefont .})
\setdashes < 0.1270cm>
\plot 10.160 16.510  6.858 22.225 /
%
%
\linethickness= 0.500pt
\setplotsymbol ({\thinlinefont .})
\plot 10.160 16.510 13.462 22.225 /
%
%
\linethickness= 0.500pt
\setplotsymbol ({\thinlinefont .})
\setsolid
\putrule from 12.764 20.993 to  7.563 20.993
%
%
\linethickness= 0.500pt
\setplotsymbol ({\thinlinefont .})
\plot  7.576 20.980  5.080 16.510 /
%
%
\linethickness= 0.500pt
\setplotsymbol ({\thinlinefont .})
\plot 12.744 20.974 15.240 16.510 /
%
%
\put{\SetFigFont{8}{14.4}{\rmdefault}{\mddefault}{\updefault}$T^-$} [lB] at  7.366 17.780
%
%
\put{\SetFigFont{8}{14.4}{\rmdefault}{\mddefault}{\updefault}$T^\bot$} [lB] at 10.065 19.082
%
%
\put{\SetFigFont{8}{14.4}{\rmdefault}{\mddefault}{\updefault}$T^+$} [lB] at 12.541 17.780
%
%
\put{\SetFigFont{8}{14.4}{\rmdefault}{\mddefault}{\updefault}$O$} [lB] at 9.9 15.9
%
%
\put{\SetFigFont{8}{14.4}{\rmdefault}{\mddefault}{\updefault}$[O,W^+)$} [lB] at 16.5 15.9
%
%
\put{\SetFigFont{8}{14.4}{\rmdefault}{\mddefault}{\updefault}$[O,W^-)$} [lB] at  1.6 15.9
%
%
\put{\SetFigFont{8}{14.4}{\rmdefault}{\mddefault}{\updefault}$W$} [lB] at  2.635 16.7
%
%
\put{\SetFigFont{8}{14.4}{\rmdefault}{\mddefault}{\updefault}$W$} [lB] at 17 16.7
\linethickness=0pt
\putrectangle corners at  2.515 22.250 and 17.805 15.767
\endpicture}}
\caption{}
\end{figure}

Let $\fS_n^W=\fS^+_n\cup \fS^-_n\cup\fS^\bot_n$.  Then 
\begin{equation}\label{freeeqn1} 
\Sigma\subset\bigcup_{T\in\fS^W_n}\Omega_T\,. 
\end{equation} 
The first step is to calculate the number of triangles in $\fS_n^W$. 
 
To do this, the number of possible choices for
$T^+\in\fS^+_n$ must be determined.  Refer to Figure \ref{freefig4} below. 
 
\refstepcounter{picture}
\begin{figure}[htbp]\label{freefig4}
\font\thinlinefont=cmr5
\begingroup\makeatletter\ifx\SetFigFont\undefined%
\gdef\SetFigFont#1#2#3#4#5{%
  \reset@font\fontsize{#1}{#2pt}%
  \fontfamily{#3}\fontseries{#4}\fontshape{#5}%
  \selectfont}%
\fi\endgroup%
\centerline{\mbox{\beginpicture
\setcoordinatesystem units <.300000cm,.300000cm>
\unitlength=.300000cm
\linethickness=1pt
\setplotsymbol ({\makebox(0,0)[l]{\tencirc\symbol{'160}}})
\setshadesymbol ({\thinlinefont .})
\setlinear
%
%
\linethickness= 0.500pt
\setplotsymbol ({\thinlinefont .})
\plot  3.810 15.240  2.709 13.335 /
\putrule from  2.709 13.335 to  4.911 13.335
\plot  4.911 13.335  3.810 15.240 /
%
%
\linethickness= 0.500pt
\setplotsymbol ({\thinlinefont .})
\plot  6.009 15.240  3.810 15.242 /
\plot  3.810 15.242  4.911 13.335 /
\plot  4.911 13.335  6.009 15.240 /
%
%
\linethickness= 0.500pt
\setplotsymbol ({\thinlinefont .})
\plot  7.110 13.337  6.011 15.242 /
\plot  6.011 15.242  4.911 13.335 /
\plot  4.911 13.335  7.110 13.337 /
%
%
\linethickness= 0.500pt
\setplotsymbol ({\thinlinefont .})
\plot  4.911 13.339  6.011 15.242 /
\putrule from  6.011 15.242 to  3.810 15.242
\plot  3.810 15.242  4.911 13.339 /
%
%
\linethickness= 0.500pt
\setplotsymbol ({\thinlinefont .})
\plot  3.812 15.244  6.011 15.242 /
\plot  6.011 15.242  4.911 17.149 /
\plot  4.911 17.149  3.812 15.244 /
%
%
\linethickness= 0.500pt
\setplotsymbol ({\thinlinefont .})
\plot  4.913 17.147  6.011 15.242 /
\plot  6.011 15.242  7.112 17.149 /
\plot  7.112 17.149  4.913 17.147 /
%
%
\linethickness= 0.500pt
\setplotsymbol ({\thinlinefont .})
\plot  7.112 17.145  6.011 15.242 /
\putrule from  6.011 15.242 to  8.213 15.242
\plot  8.213 15.242  7.112 17.145 /
%
%
\linethickness= 0.500pt
\setplotsymbol ({\thinlinefont .})
\plot  8.211 15.240  6.011 15.242 /
\plot  6.011 15.242  7.112 13.335 /
\plot  7.112 13.335  8.211 15.240 /
%
%
\linethickness= 0.500pt
\setplotsymbol ({\thinlinefont .})
\plot  7.112 17.145  4.913 17.147 /
\plot  4.913 17.147  6.013 15.240 /
\plot  6.013 15.240  7.112 17.145 /
%
%
\linethickness= 0.500pt
\setplotsymbol ({\thinlinefont .})
\plot  7.112 17.145  6.013 19.050 /
\plot  6.013 19.050  4.913 17.143 /
\plot  4.913 17.143  7.112 17.145 /
%
%
\linethickness= 0.500pt
\setplotsymbol ({\thinlinefont .})
\plot  7.112 17.145  8.213 19.048 /
\putrule from  8.213 19.048 to  6.011 19.048
\plot  6.011 19.048  7.112 17.145 /
%
%
\linethickness= 0.500pt
\setplotsymbol ({\thinlinefont .})
\plot  7.112 17.145  9.311 17.143 /
\plot  9.311 17.143  8.211 19.050 /
\plot  8.211 19.050  7.112 17.145 /
%
%
\linethickness= 0.500pt
\setplotsymbol ({\thinlinefont .})
\plot  7.112 17.145  8.211 15.240 /
\plot  8.211 15.240  9.311 17.147 /
\plot  9.311 17.147  7.112 17.145 /
%
%
\linethickness= 0.500pt
\setplotsymbol ({\thinlinefont .})
\plot  9.311 17.143  8.211 15.240 /
\putrule from  8.211 15.240 to 10.412 15.240
\plot 10.412 15.240  9.311 17.143 /
%
%
\linethickness= 0.500pt
\setplotsymbol ({\thinlinefont .})
\plot 10.410 15.238  8.211 15.240 /
\plot  8.211 15.240  9.311 13.333 /
\plot  9.311 13.333 10.410 15.238 /
%
%
\linethickness= 0.500pt
\setplotsymbol ({\thinlinefont .})
\plot  9.309 13.335  8.211 15.240 /
\plot  8.211 15.240  7.110 13.333 /
\plot  7.110 13.333  9.309 13.335 /
%
%
\linethickness= 0.500pt
\setplotsymbol ({\thinlinefont .})
\plot 10.410 15.238  9.311 17.143 /
\plot  9.311 17.143  8.211 15.236 /
\plot  8.211 15.236 10.410 15.238 /
%
%
\linethickness= 0.500pt
\setplotsymbol ({\thinlinefont .})
\plot 10.410 15.238 11.510 17.141 /
\putrule from 11.510 17.141 to  9.309 17.141
\plot  9.309 17.141 10.410 15.238 /
%
%
\linethickness= 0.500pt
\setplotsymbol ({\thinlinefont .})
\plot 10.410 15.238 12.609 15.236 /
\plot 12.609 15.236 11.508 17.143 /
\plot 11.508 17.143 10.410 15.238 /
%
%
\linethickness= 0.500pt
\setplotsymbol ({\thinlinefont .})
\plot 10.410 15.238 11.508 13.333 /
\plot 11.508 13.333 12.609 15.240 /
\plot 12.609 15.240 10.410 15.238 /
%
%
\linethickness= 0.500pt
\setplotsymbol ({\thinlinefont .})
\plot 10.410 15.238  9.309 13.335 /
\putrule from  9.309 13.335 to 11.510 13.335
\plot 11.510 13.335 10.410 15.238 /
%
%
\linethickness= 0.500pt
\setplotsymbol ({\thinlinefont .})
\plot 11.508 17.143 10.408 15.240 /
\putrule from 10.408 15.240 to 12.609 15.240
\plot 12.609 15.240 11.508 17.143 /
%
%
\linethickness= 0.500pt
\setplotsymbol ({\thinlinefont .})
\plot 13.708 17.145 11.508 17.147 /
\plot 11.508 17.147 12.609 15.240 /
\plot 12.609 15.240 13.708 17.145 /
%
%
\linethickness= 0.500pt
\setplotsymbol ({\thinlinefont .})
\plot 14.808 15.242 13.710 17.147 /
\plot 13.710 17.147 12.609 15.240 /
\plot 12.609 15.240 14.808 15.242 /
%
%
\linethickness= 0.500pt
\setplotsymbol ({\thinlinefont .})
\plot 13.710 13.337 14.810 15.240 /
\putrule from 14.810 15.240 to 12.609 15.240
\plot 12.609 15.240 13.710 13.337 /
%
%
\linethickness= 0.500pt
\setplotsymbol ({\thinlinefont .})
\plot 11.510 13.335 13.710 13.333 /
\plot 13.710 13.333 12.609 15.240 /
\plot 12.609 15.240 11.510 13.335 /
%
%
\linethickness= 0.500pt
\setplotsymbol ({\thinlinefont .})
\plot 15.909 13.335 14.810 15.240 /
\plot 14.810 15.240 13.710 13.333 /
\plot 13.710 13.333 15.909 13.335 /
%
%
\linethickness= 0.500pt
\setplotsymbol ({\thinlinefont .})
\plot 12.609 15.244 13.710 17.147 /
\putrule from 13.710 17.147 to 11.508 17.147
\plot 11.508 17.147 12.609 15.244 /
%
%
\linethickness= 0.500pt
\setplotsymbol ({\thinlinefont .})
\plot 11.510 17.149 13.710 17.147 /
\plot 13.710 17.147 12.609 19.054 /
\plot 12.609 19.054 11.510 17.149 /
%
%
\linethickness= 0.500pt
\setplotsymbol ({\thinlinefont .})
\plot 11.510 17.149 12.609 15.244 /
\plot 12.609 15.244 13.710 17.151 /
\plot 13.710 17.151 11.510 17.149 /
%
%
\linethickness= 0.500pt
\setplotsymbol ({\thinlinefont .})
\plot 11.510 17.149 10.410 15.246 /
\putrule from 10.410 15.246 to 12.611 15.246
\plot 12.611 15.246 11.510 17.149 /
%
%
\linethickness= 0.500pt
\setplotsymbol ({\thinlinefont .})
\plot 11.510 17.149  9.311 17.151 /
\plot  9.311 17.151 10.412 15.244 /
\plot 10.412 15.244 11.510 17.149 /
%
%
\linethickness= 0.500pt
\setplotsymbol ({\thinlinefont .})
\plot 11.510 17.149 10.412 19.054 /
\plot 10.412 19.054  9.311 17.147 /
\plot  9.311 17.147 11.510 17.149 /
%
%
\linethickness= 0.500pt
\setplotsymbol ({\thinlinefont .})
\plot 11.510 17.149 12.611 19.052 /
\putrule from 12.611 19.052 to 10.410 19.052
\plot 10.410 19.052 11.510 17.149 /
%
%
\linethickness= 0.500pt
\setplotsymbol ({\thinlinefont .})
\plot  9.311 17.147 11.510 17.145 /
\plot 11.510 17.145 10.410 19.052 /
\plot 10.410 19.052  9.311 17.147 /
%
%
\linethickness= 0.500pt
\setplotsymbol ({\thinlinefont .})
\plot  8.211 19.050  9.309 17.145 /
\plot  9.309 17.145 10.410 19.052 /
\plot 10.410 19.052  8.211 19.050 /
%
%
\linethickness= 0.500pt
\setplotsymbol ({\thinlinefont .})
\plot  9.309 17.145 10.410 19.048 /
\putrule from 10.410 19.048 to  8.208 19.048
\plot  8.208 19.048  9.309 17.145 /
%
%
\linethickness= 0.500pt
\setplotsymbol ({\thinlinefont .})
\plot 11.508 17.143 10.410 19.048 /
\plot 10.410 19.048  9.309 17.141 /
\plot  9.309 17.141 11.508 17.143 /
%
%
\linethickness= 0.500pt
\setplotsymbol ({\thinlinefont .})
\plot 12.609 19.046 10.410 19.048 /
\plot 10.410 19.048 11.510 17.141 /
\plot 11.510 17.141 12.609 19.046 /
%
%
\linethickness= 0.500pt
\setplotsymbol ({\thinlinefont .})
\plot  9.313 24.761  8.213 22.858 /
\putrule from  8.213 22.858 to 10.414 22.858
\plot 10.414 22.858  9.313 24.761 /
%
%
\linethickness= 0.500pt
\setplotsymbol ({\thinlinefont .})
\plot  2.716 13.350 10.312 26.490 /
%
%
\linethickness= 0.500pt
\setplotsymbol ({\thinlinefont .})
\plot  9.315 24.746 12.617 19.044 /
%
%
\linethickness= 0.500pt
\setplotsymbol ({\thinlinefont .})
\putrule from 14.605 13.335 to 21.431 13.335
%
%
\put{\SetFigFont{8}{14.4}{\rmdefault}{\mddefault}{\updefault}$[O,W^+)$} [lB] at 18 12.2
%
%
\put{\SetFigFont{8}{14.4}{\rmdefault}{\mddefault}{\updefault}$c_1$} [lB] at  3.5 13.80
%
%
\put{\SetFigFont{8}{14.4}{\rmdefault}{\mddefault}{\updefault}1} [lB] at  4.7 14.25
%
%
\put{\SetFigFont{8}{14.4}{\rmdefault}{\mddefault}{\updefault}1} [lB] at  5.8 13.8
%
%
\put{\SetFigFont{8}{14.4}{\rmdefault}{\mddefault}{\updefault}1} [lB] at  6.9 14.25
%
%
\put{\SetFigFont{8}{14.4}{\rmdefault}{\mddefault}{\updefault}1} [lB] at  7.99 13.8
%
%
\put{\SetFigFont{8}{14.4}{\rmdefault}{\mddefault}{\updefault}1} [lB] at  9.09 14.25
%
%
\put{\SetFigFont{8}{14.4}{\rmdefault}{\mddefault}{\updefault}1} [lB] at 10.18 13.8
%
%
\put{\SetFigFont{8}{14.4}{\rmdefault}{\mddefault}{\updefault}1} [lB] at 11.28 14.25
%
%
\put{\SetFigFont{8}{14.4}{\rmdefault}{\mddefault}{\updefault}1} [lB] at 12.37 13.8
%
%
\put{\SetFigFont{8}{14.4}{\rmdefault}{\mddefault}{\updefault}1} [lB] at 13.47 14.25
%
%
\put{\SetFigFont{8}{14.4}{\rmdefault}{\mddefault}{\updefault}1} [lB] at 14.56 13.8
%
%
\put{\SetFigFont{8}{14.4}{\rmdefault}{\mddefault}{\updefault}$c_2$} [lB] at  4.52 15.7
%
%
\put{\SetFigFont{8}{14.4}{\rmdefault}{\mddefault}{\updefault}2} [lB] at  5.8 16.1
%
%
\put{\SetFigFont{8}{14.4}{\rmdefault}{\mddefault}{\updefault}2} [lB] at  6.9 15.7
%
%
\put{\SetFigFont{8}{14.4}{\rmdefault}{\mddefault}{\updefault}2} [lB] at  7.99 16.1
%
%
\put{\SetFigFont{8}{14.4}{\rmdefault}{\mddefault}{\updefault}2} [lB] at  9.09 15.7
%
%
\put{\SetFigFont{8}{14.4}{\rmdefault}{\mddefault}{\updefault}2} [lB] at 10.18 16.1
%
%
\put{\SetFigFont{8}{14.4}{\rmdefault}{\mddefault}{\updefault}2} [lB] at 11.28 15.7
%
%
\put{\SetFigFont{8}{14.4}{\rmdefault}{\mddefault}{\updefault}2} [lB] at 12.37 16.1
%
%
\put{\SetFigFont{8}{14.4}{\rmdefault}{\mddefault}{\updefault}2} [lB] at 13.47 15.7
%
%
\put{\SetFigFont{8}{14.4}{\rmdefault}{\mddefault}{\updefault}$c_n$} [lB] at  8.95 23.3
%
%
\put{\SetFigFont{8}{14.4}{\rmdefault}{\mddefault}{\updefault}$O$} [lB] at  2.223 12.2
\linethickness=0pt
\putrectangle corners at  2.223 26.515 and 21.457 12.624
\endpicture}}
\hfil{}
\caption{}
\end{figure}

There are $(q+1)$ possible choices for the chamber $c_1$.  This 
choice then determines all the other chambers labeled 1 which 
lie in the convex hull of $c_1$ and $[0,W^+)$.  There are then 
$q$ choices for the chamber $c_2$.  This choice now determines all  
chambers labeled 2 which are in the convex hull of $c_2$ and 
all the other chambers previously determined. 
 
Continue in this way until the chamber $c_n$ is reached, thereby 
determining the whole triangle $T^+\in\fS_n^+$.  There are thus 
$(q+1)q^{n-1}$ choices for $T^+$. 
 
Now each triangle $T^+\in\fS_n^+$ determines a unique pair of triangles 
$T^-\in\fS_n^-$, $T^\bot\in\fS_n^\bot$ subject to the condition that 
$T^-$, $T^\bot$ lie in the convex hull of $\fS^+_n\cup W$.  Conversely  
such $T^-$ or $T^\bot$ determine $T^+$ uniquely.  Hence the sets $\fS^+_n$, 
$\fS_n^-$, $\fS_n^\bot$ have the same number of elements.  It follows 
that  
$$|\fS_n^W|=3(q+1)q^{n-1}\,.$$ 
 
Since the sets $\Omega_T$, $T\in\fS_n$ have equal measure and  
partition $\Omega$, it follows from Lemma \ref{free1} and  
equation (\ref{freeeqn1}) that 
\begin{eqnarray*} 
\nu_O(\Sigma) 
&\leq&\nu_O\left(\bigcup_{T\in\fS_n^W}\Omega_T\right)\\ 
&=&\frac{3(q+1)q^{n-1}}{(q^2+q+1)(q+1)q^{3n-3}}\\ 
&\to&0\qquad\hbox{as $n\to\infty$}\,. 
\end{eqnarray*} 
Thus $\nu_O(\Sigma)=0$. 
\qed

\end{document}